\documentclass[12pt,reqno]{amsart}
\usepackage{geometry}
\usepackage{amssymb}
\usepackage{amsmath,amstext,amssymb,amscd}
\usepackage{verbatim}
\usepackage{enumerate}
\usepackage{mathrsfs}
\usepackage{bm}
\usepackage{amsthm}
\usepackage[colorlinks=true,linkcolor=black,citecolor=black]{hyperref}
\usepackage{color}
\usepackage{setspace}
\usepackage{enumitem}
\newtheorem{proposition}{Proposition}[section]
\newtheorem{theorem}[proposition]{Theorem}

\newtheorem{lemma}[proposition]{Lemma}
\theoremstyle{definition}
\newtheorem{definition}[proposition]{Definition}
\newtheorem{remark}[proposition]{Remark}

\numberwithin{equation}{section}
\title[Smoothness of the Inertial Manifold]{Smoothness of inertial manifold for the Burgers equation}
\author[Ziqi NiU, ~ Xinhua Li, ~ Chunyou Sun, ~ and Xiaoqing Yang]{Ziqi Niu, ~ Xinhua Li, ~ Chunyou Sun${}^{\ast}$, ~ and Xiaoqing Yang}
\address{School of Mathematics and Statistics, Lanzhou University, Lanzhou, 730000, P.R. China}
\email{niuzq21@lzu.edu.cn (Z.Niu)}
\address{School of Mathematics and Statistics, Lanzhou University, Lanzhou, 730000, P.R. China}
\email{lxh@lzu.edu.cn (X.Li)}
\address{School of Mathematics and Statistics, Lanzhou University, Lanzhou, 730000, P.R. China}
\address{School of Mathematics and Statistics, Donghua University, Shanghai, 201620, P.R. China}
\email{sunchy@lzu.edu.cn (C.Sun)}
\address{School of Mathematics and Statistics, Lanzhou University, Lanzhou, 730000, P.R. China}
\email{xiaoqingyang22@lzu.edu.cn (X.Yang)}
\begin{document}
	\begin{abstract}
		This paper establishes a ${C^{n,\varepsilon }}$-smooth extension of the inertial manifold for the one-dimensional Burgers equation, which demonstrates that its long-time behavior can be completely determined by explicit smooth first-order ODEs. We first devise a new framework for an abstract equation with two nonlinear terms, where one preserves regularity and the other reduces regularity, and derive sufficient conditions for constructing the ${C^{n,\varepsilon}}$-smooth extension of the IM by treating these two nonlinear terms separately.
	\end{abstract}
	\par
	\thanks{${}^\ast$ Corresponding author}
	\thanks{This work was partially supported by the National Natural Science Foundation of China (Grants no. 12271227 and 12201269).}
	\maketitle
	\vspace{-0.6cm}
	\begin{spacing}{1}
		\tableofcontents
	\end{spacing}
	\vspace{-0.6cm}
	\section{Introduction}\label{Introduction}
	\noindent
	\par
	This paper establishes the ${C^{n,\varepsilon }}$-smooth extension of an inertial manifold (IM) for the one-dimensional Burgers equation in a real Hilbert space $H_0^1$:
	\begin{equation}\label{main equation}
		\begin{array}{*{20}{c}}
			{{\partial _t}u - \nu {\partial _{xx}}u = u{\partial _x}u + g(x),}&{x \in \Omega = (0,\pi ),}&{u{|_{x = 0}} = u{|_{x = \pi }} = 0,}
		\end{array}
	\end{equation}
	where $u(x,t)$ denotes the unknown velocity function, $\nu$ represents the kinematic viscosity, and $g$ is the external force term. 
	\par
	The theoretical foundations of the Burgers equation trace back to the seminal 1948 research by a pioneer in fluid dynamics, J.M. Burgers \cite{B1948}, who revealed the fundamental mechanism of turbulence through the interaction between the convective term ${u{\partial _x}u}$ and dissipation term ${\nu {\partial _{xx}}u}$ (see, e.g., \cite{BPZ2014, LPP2021, L2018}). Beyond turbulence analysis, the Burgers equation serves as a fundamental theoretical model in diverse physical systems, including nonlinear wave dynamics, colloidal sedimentation processes, interfacial growth in molecular systems, and elastic wave propagation in isotropic solids, among others.
	\par
	The smooth IM is a finite-dimensional smooth invariant submanifold in phase space with an exponential tracking property (see, e.g., \cite{LS2020, ZS2014}). If a smooth IM exists for Partial Differential Equation (PDE) \eqref{main equation}, its long-time behavior can be completely determined by the following finite-dimensional Ordinary Differential Equations (ODEs):
	\begin{equation}\label{IF of the main equation}
		\begin{array}{*{20}{c}}
			{\frac{d}{{dt}}{u_N} - \nu {\partial _{xx}}{u_N} = {P_N}(({u_N} + M({u_N})){\partial _x}({u_N} + M({u_N})) + g(x)),}&{{u_N} \in {P_N}(H_0^1),}
		\end{array}
	\end{equation}
	where the operator ${ - {\partial _{xx}}}$ possesses a discrete spectrum in $H_0^1$ with eigenvalues $0 < {\lambda _1} < {\lambda _2} < \cdots $, and its corresponding orthonormal eigenfunctions $\{ {e_n}\} _{n = 1}^\infty $ form a complete basis of $H_0^1$. The projection operator ${P_N}$ maps $H_0^1$ onto the $N$-dimensional subspace ${P_N}(H_0^1): = \text{span}\{ {e_1},{e_2}, \cdot \cdot \cdot ,{e_N}\} \sim{\mathbb{R}^N}$, while $M:{P_N}(H_0^1) \to {Q_N}(H_0^1): = {({P_N}(H_0^1))^ \bot }$ is the map whose graph defines the ${C^n}$-smooth IM. The system \eqref{IF of the main equation} is called the ${C^n}$-smooth inertial form (IF) of the original dissipative equation \eqref{main equation} (see, e.g., \cite{BLZ1999, BLZ2008, CCLP2013, CL2007, CLR2010, CLMO2025, HGT2015, TR1997, ZS2014, ZSL2020} and the references therein). The existence of smooth IF holds significant importance in theoretical analysis and practical applications: On the one hand, it enables the application of classical dynamical methods from ODE theory to analyzing more detailed dynamical behavior of the PDE (for instance, the simplest saddle-node bifurcation requires ${C^{2}}$-smoothness, and the Hopf bifurcation needs ${C^{3}}$-smoothness, see, e.g., \cite{KH1995, KH2012}). On the other hand, the ${C^n}$-smooth IF significantly reduces the computational complexity of numerical simulations (this reduction requires the IM to have higher smoothness). Thus, establishing the existence of smooth IM/IF has become an important topic in the field of dissipative PDEs (see, e.g., \cite{CL1988a, FN1972, FNST1988, MS1988}).
	\par
	However, due to the convective term in the Burgers equation \eqref{main equation}, even constructing a ${C^{1,\varepsilon}}$-smooth IM/IF poses significant challenges. Specifically, establishing a ${C^{1,\varepsilon}}$-smooth IM for the Burgers equation \eqref{main equation} requires the following spectral gap condition to hold:
	\begin{equation}
		\begin{array}{*{20}{c}}
			{\frac{{\nu {\lambda _{N + 1}} - \nu {\lambda _N}}}{{{{(\nu {\lambda _{N + 1}})}^{1/2}} + {{(\nu {\lambda _N})}^{1/2}}}} > L,}&{\exists N \in {\mathbb{N}^ + },}
		\end{array}
	\end{equation}
	where $L$ is the Lipschitz constant of the nonlinearity. Clearly, if no additional limitation on the viscosity coefficient $\nu$, this spectral gap condition cannot be satisfied generally, making it impossible to construct a ${C^{1,\varepsilon}}$-smooth IM/IF for the Burgers equation directly. To address this challenge, researchers have developed an innovative approach that transforms the original equation into a new form or embeds it into a larger system, ensuring that the new equation satisfies the spectral gap condition and thus enables the construction of the IM/IF. Since the transformed equation is equivalent to the original one, the resulting IM is also regarded as an IM for the original equation. This methodology has been widely applied to mathematical and physical models (see, e.g., \cite{SY1992, V2008, V2009} and the references therein). For the Burgers-type equation, in 2011, J. Vukadinovic \cite{V2011} employed the Cole-Hopf transformation to convert the diffusive Burgers equation with low-wavenumber into a new form, and further constructed a Lipschitz-smooth IM/IF. Although the transformed equation does not contain convective terms, the Cole-Hopf transformation is merely a Lipschitz diffeomorphism, and thus can only be used to construct a Lipschitz-smooth IM/IF. In 2017, A. Kostianko and S. Zelik \cite{KZ2017, KZ2018} developed a smooth auxiliary transformation and constructed the ${C^{1,\varepsilon }}$-smooth IM/IF for reaction-diffusion-advection systems.
	\par
	Using the aforementioned diffeomorphic transformation, we can transform the Burgers equation \eqref{main equation} (with $\nu = 1$) into a new form. Specifically, we consider the transformation $u(x,t) = a(x,t)v(x,t)$, where $a(x,t) \in \mathbb{R}\backslash \{ 0\} $. Substituting this into the original Burgers equation \eqref{main equation} yields:
	\begin{equation}\label{transformed Burgers equation}
		{\partial _t}v - \partial _x^2v = {a^{ - 1}}[ava + 2{\partial _x}a]{\partial _x}v + \{ [{a^{ - 1}}\partial _x^2a - {a^{ - 1}}{\partial _t}a + {a^{ - 1}}av{\partial _x}a]v + {a^{ - 1}}g(x)\} .
	\end{equation}
	Based on this, the construction of smoother IM/IF can be further investigated. 
	\par
	Currently, there are two primary methods to investigate the smoothness of IMs: The first method, developed by S.-N. Chow, K. Lu, G. Sell, S. Zelik et al. (see, e.g., \cite{CL1988b, CLS1992, KLSZ2022, NLS2026, ZS2014}), directly constructs smooth IMs via the spectral gap condition. Taking the transformed Burgers equation \eqref{transformed Burgers equation} as an example, establishing a ${C^{n,\varepsilon }}$-smooth IM requires satisfying the following spectral gap condition:
	\begin{equation}\label{exponential spectral gap condition}
		\begin{array}{*{20}{c}}
			{\frac{{{\lambda _{N + 1}} - n{\lambda _N}}}{{\lambda _{N + 1}^{1/2} + \lambda _N^{1/2}}} > C(n)L,}&{\exists N \in {\mathbb{N}^ + },}
		\end{array}
	\end{equation}
	where the constant $C(n)$ depends on $n$, and $L$ is the Lipschitz constant of the nonlinearity in equation \eqref{transformed Burgers equation}, and $\varepsilon $ is a small parameter depending on the spectral gap and $n$. The second method, introduced by A. Kostianko and S. Zelik \cite{KZ2024}, employs the Whitney Extension Theorem to construct a ${C^{n,\varepsilon }}$-smooth extended submanifold ${\widetilde{\mathcal M}_{{N_n}}}$, which is an extension of the ${C^{1,\varepsilon }}$-smooth initial IM ${{\mathcal M}_{{N_1}}}$ (the definition of the ${C^{n,\varepsilon }}$-smooth extension ${\widetilde{\mathcal M}_{{N_n}}}$ will be recalled in Section \ref{Section 4}), thereby relaxing the aforementioned spectral gap condition and using spectral projection to obtain the desired ${C^{n,\varepsilon }}$-smooth IF for the original system. For equation \eqref{transformed Burgers equation}, if we apply directly the method/framework in \cite{KZ2024}, constructing a ${C^{n,\varepsilon}}$-smooth extension of IM requires the following conditions:
	\begin{equation}\label{Zelik spectral gap condition}
		\begin{array}{*{20}{c}}
			{\frac{{{\lambda _{{N_i} + 1}} - {\lambda _{{N_i}}} - {C_1}(n){\lambda _{{N_{i - 1}}}}}}{{\lambda _{{N_i} + 1}^{1/2} + \lambda _{{N_i}}^{1/2}}} > {C_2}(n)L,}&{i = 1,2, \cdot \cdot \cdot n,}&{{\lambda _{{N_0}}} = 0.}
		\end{array}
	\end{equation}
	\par
	Clearly, the transformed Burgers equation \eqref{transformed Burgers equation} does not satisfy the aforementioned smoothness conditions \eqref{exponential spectral gap condition} and \eqref{Zelik spectral gap condition}. The fundamental challenge is that the nonlinear term in the transformed Burgers equation still reduces regularity, while existing methods for constructing ${C^{n,\varepsilon}}$-smooth IM or ${C^{n,\varepsilon}}$-smooth extension of IM treat nonlinear terms as a whole and can only handle the cases with a single nonlinear term that preserves regularity, thus being inapplicable to the (original or transformed) Burgers equation. Hence, whether a smooth IM/IF exists for the Burgers equation remains unknown.
	\par
	In this paper, to solve the difficulty mentioned above, we first focus on an abstract equation with two nonlinear terms, where one preserves regularity and the other reduces regularity. By treating these two nonlinear terms separately, we derive sufficient conditions for constructing the ${C^{n,\varepsilon}}$-smooth extension of the IM. Specifically, we consider the following abstract parabolic problem:
	\begin{equation}\label{abstract parabolic problem}
		\begin{array}{*{20}{c}}
			{{u_t} + Au = {F_1}(u) + {F_2}(u),}&{{{\left. u \right|}_{t = 0}} = {u_0},}&{{F_1}:H_0^1 \to H,}&{{F_2}:H_0^1 \to H_0^1,}
		\end{array}
	\end{equation}
	where $A$ is a positive self-adjoint operator with compact inverse, and the nonlinear terms satisfy ${F_1} \in C_b^{n + 1}(H_0^1,H)$ and ${F_2} \in C_b^{n + 1}(H_0^1,H_0^1)$. Let $0 < {\lambda _1} \le {\lambda _2} \le \cdots $ be the eigenvalues of $A$ enumerated in non-decreasing order, and ${L_1}$, ${L_2}$ respectively denote the Lipschitz constants of the nonlinear terms ${F_1}$, ${F_2}$. Our first main result is the following abstract theorem, which provides the spectral gap condition for constructing the ${C^{n,\varepsilon }}$-smooth extension of the IM for equation \eqref{abstract parabolic problem}.
	\begin{theorem}\label{my theorem}
		For any $n \in {\mathbb{N}^ + }$ and sufficiently small $\varepsilon \in (0,1)$, let ${F_1} \in C_b^{n + 1}(H_0^1,H)$ and ${F_2} \in C_b^{n + 1}(H_0^1,H_0^1)$. If there exists a family $\{ {\lambda _{{N_i}}}\} _{i = 1}^n$ satisfying the spectral gap condition below:
		\begin{equation}\label{my spectral gap condition}
			\begin{array}{*{20}{c}}
				{{\lambda _{{N_i} + 1}} - {\lambda _{{N_i}}} > (i - 1){\lambda _{{N_{i - 1}}}} + \lambda _{{N_i} + 1}^{1/2}(i + 1){L_1} + (i + 1){L_2},}&{{\lambda _{{N_0}}} = 0,}
			\end{array}
		\end{equation}
		where ${L_1}$ and ${L_2}$ respectively denote the Lipschitz constants of the nonlinear terms ${F_1}$ and ${F_2}$.
		\par
		Then equation \eqref{abstract parabolic problem} possesses a ${C^{n,\varepsilon }}$-smooth IF, as well as an $N_n$-dimensional ${C^{n,\varepsilon }}$-smooth extension ${\widetilde{\mathcal M}_{{N_n}}}$ of the IM ${{\mathcal M}_{{N_1}}}$.
	\end{theorem}
	\par
	Next, we consider the Burgers equation \eqref{main equation}. First, we transform the Burgers equation into an equivalent form via a diffeomorphism. The transformed equation has two nontrivial nonlinear terms, namely ${{\mathcal I}_1}: H_0^1 \to H$ and ${{\mathcal I}_2}: H_0^1 \to H_0^1$, and we verify that ${{\mathcal I}_1} \in C_b^{n+1}(H_0^1, H)$ and ${{\mathcal I}_2} \in C_b^{n+1}(H_0^1, H_0^1)$. Based on Theorem \ref{my theorem}, we further prove that the transformed equation possesses a family $\{ {\lambda _{{N_i}}}\} _{i = 1}^n$ satisfying the spectral gap condition \eqref{my spectral gap condition}, and therefore establish the ${C^{n,\varepsilon}}$-smooth IF as well as a ${C^{n,\varepsilon}}$-smooth extension of the IM for equation \eqref{transformed Burgers equation}. Since the transformed equation is equivalent to the original Burgers equation \eqref{main equation}, we indeed establish the following result:
	\par
	\begin{theorem}\label{main theorem}
		For any $n \in {\mathbb{N}^ + }$ and sufficiently small $\varepsilon \in (0,1)$, the Burgers equation \eqref{main equation} possesses the ${C^{n,\varepsilon }}$-smooth IF, as well as an $N_n$-dimensional ${C^{n,\varepsilon }}$-smooth extension ${\widetilde{\mathcal M}_{{N_n}}}$ of the IM ${{\mathcal M}_{{N_1}}}$.
	\end{theorem}
	Theorem \ref{main theorem} indicates that even if the smoothness conditions \eqref{exponential spectral gap condition} or \eqref{Zelik spectral gap condition} are not satisfied, some equations involving convective terms (such as reaction-diffusion-advection equations) can still possess the smooth IF. It further demonstrates that diffeomorphic transformations show promise in the construction of smooth IF and smooth extension of IM.
	\par
	The paper is organized as follows.
	\par
	In Section \ref{Section 2}, we review standard facts about smooth functions in Banach spaces, including the Taylor Theorem, the Converse Taylor Theorem, and the Whitney Extension Theorem, all of which serve as crucial technical tools in subsequent analyses. In Section \ref{Section 3}, we collect basic facts about the construction of IM for the one-dimensional Burgers equation. The proof of Theorem \ref{my theorem} is provided in Section \ref{Section 4}. In Section \ref{Section 5}, Theorem \ref{my theorem} is applied to the transformed Burgers equation, thereby proving Theorem \ref{main theorem}.
	
	\section{Preliminaries: Taylor Theorem and Whitney Extension Theorem}\label{Section 2}
	\noindent
	\par
	This section reviews standard results on Taylor expansions of smooth functions in Banach spaces, as well as how to use the Converse Taylor Theorem and Whitney Extension Theorem to prove the smoothness of functions, which serve as technical tools for subsequent analysis (see \cite{HJ2014} for details).
	\par
	First, we review multilinear maps and polynomials between normed linear spaces. Let $X$, $Y$ be normed spaces and $n \in \mathbb{N}^+$. Define ${\mathcal L}_s(X^n,Y)$ as the space of continuous symmetric multilinear maps from ${X^n}$ to $Y$, equipped with the norm:
	\begin{equation}
		{\| Q \|_{{{\mathcal L}_s}( {{X^n},Y} )}}: = \mathop {\sup }\limits_{{\xi _i} \in X,{\xi _i} \ne 0}\{ {\frac{{{{\| {Q( {{\xi _1}, \cdot \cdot \cdot ,{\xi _n}} )} \|}_Y}}}{{{{\| {{\xi _1}} \|}_X} \cdot \cdot \cdot {{\| {{\xi _n}} \|}_X}}}}\}.
	\end{equation}
	\par
	For any $Q \in {{\mathcal L}_s}( {{X^n},Y} )$, we define a homogeneous continuous polynomial ${P_Q}$ of order $n$ on $X$ (with values in $Y$) ${P_Q}(\xi ): = Q(\xi , \cdot \cdot \cdot ,\xi )$ (with $\xi$ repeated $n$ times inside the parentheses). Conversely, given a homogeneous polynomial, the corresponding symmetric multilinear map $Q = {Q_P}$ can be uniquely recovered via the Polarization Formula (see e.g., \cite{HJ2014}). Hence, there is a one-to-one correspondence between homogeneous polynomials and symmetric multilinear maps. Moreover, let ${{\mathcal P}_n}( {X,Y} )$ denote the space of $n$-homogeneous polynomials, with the norm:
	\begin{equation}
		{\| P \|_{{{\mathcal P}_n}( {X,Y} )}}: = \mathop {\sup }\limits_{\xi \ne 0} \{ {\frac{{{{\| {P(\xi )} \|}_Y}}}{{\| \xi \|_X^n}}}\},
	\end{equation}
	this correspondence becomes an isometry. Next, define ${{\mathcal P}^n}( {X,Y} )$ as the space of all continuous polynomials of order less than or equal to $n$ on $X$ with values in $Y$, i.e., $P( \cdot ) \in {{\mathcal P}^n}( {X,Y} )$ if
	\begin{equation*}
		\begin{array}{*{20}{c}}
			{P( \cdot ) = \sum\limits_{j = 0}^n {P_j}( \cdot ), }&{{P_j}( \cdot ) \in {{\mathcal P}_j}(X,Y).}
		\end{array}
	\end{equation*}
	\par
	Let $U \subset X$ be an open set and let $F:U \to Y$ be a map. For any $u \in U$, the Fr\'{e}chet derivative of $F$ at $u$ is denoted by $F'( u ) \in {\mathcal L}( {X,Y} )$. Likewise, for any $n \in \mathbb{N}$, its $n$th Fr\'{e}chet derivative is denoted by ${F^{( n )}}( u )$, which belongs to the space ${{\mathcal L}_s}( {{X^n},Y} )$. The set of all functions $F:U \to Y$ such that ${F^{( n )}}( u )$ exists and is continuous as a function from $U$ to ${{\mathcal L}_s}( {{X^n},Y} )$, is denoted by ${C^n}( {U,Y} )$. For any $\alpha \in (0,1)$, ${C^{n,\alpha }}( {U,Y} )$ is defined as the space of functions $F \in {C^n}( {U,Y} )$ such that ${F^{( n )}}$ is H\"{o}lder continuous with exponent $\alpha$ on $U$. The notation ${F^{( n )}}( u )[ {{\xi _1}, \cdot \cdot \cdot ,{\xi _n}} ]$ is used to represent the action of ${F^{( n )}}( u )$ on vectors $({\xi _1}, \cdot \cdot \cdot ,{\xi _n}) \in {X^n}$.
	\par
	Finally, we express the Taylor jets of length $n + 1$ of the function $F$ at point $u$ as $J_\xi ^nF( u )$:
	\begin{equation}
		\begin{array}{*{20}{c}}
			{J_\xi ^nF(u): = F(u) + \frac{1}{{1!}}F'(u)\xi + \frac{1}{{2!}}F''(u)[\xi ,\xi ] + \cdot \cdot \cdot + \frac{1}{{n!}}{F^{(n)}}(u)[{{\{ \xi \} }^n}],}&{\xi \in X.}
		\end{array}
	\end{equation}
	It is clear that, for every $u \in U$, the function $\xi \to J_\xi ^nF( u ) \in {{\mathcal P}^n}( {X,Y} )$. Similarly, the truncated Taylor jets (excluding the zeroth-order term) can be expressed as follows:
	\begin{equation}\label{truncated Taylor jets}
		\begin{array}{*{20}{c}}
			{j_\xi ^nF(u): = \frac{1}{{1!}}F'(u)\xi + \frac{1}{{2!}}F''(u)[\xi ,\xi ] + \cdot \cdot \cdot + \frac{1}{{n!}}{F^{(n)}}(u)[{{\{ \xi \} }^n}],}&{\xi \in X.}
		\end{array}
	\end{equation}
	Next, we will present some theorems about smooth functions.
	\par
	\begin{theorem}[Taylor Theorem]\label{Taylor Theorem}
		Let $F \in {C^n}( {U,Y} )$, ${u_1},{u_2} \in U$, $\xi : = {u_2} - {u_1}$ and assume ${u_t}: = t{u_1} + ( {1 - t} ){u_2} \in U$ for all $t \in [ {0,1} ]$. Then
		\begin{equation*}
			\begin{array}{*{20}{c}}
				{F({u_2}) = J_\xi ^nF({u_1}) + \frac{1}{{n!}}[\int_0^1 {{{(1 - s)}^{n - 1}}({F^{(n)}}({u_1} + s\xi ) - {F^{(n)}}({u_1}))ds][{{\{ \xi \} }^n}]} ,}&{\xi : = {u_2} - {u_1}.}
			\end{array}
		\end{equation*}
		In particular, if $F \in {C^{n,\alpha }}( {U,Y} )$, then for some positive $C$,
		\begin{equation}
			{\| {F({u_2}) - J_\xi ^nF({u_1})} \|_Y} \le C\| \xi \|_X^{n + \alpha }.
		\end{equation}
	\end{theorem}
	The aforementioned theorem can be inverted as follows.
	\par
	\begin{theorem}[Converse Taylor Theorem]
		Let $F$ be a function such that for any $u \in U$, there exists a polynomial $\xi \to P( {\xi ,u} ) \in {\mathcal{P}^n}(X,Y)$ such that, for all ${u_1},{u_2} \in U$,
		\begin{equation}\label{Converse Taylor Theorem remainder term}
			\begin{array}{*{20}{c}}
				{{{\| {F({u_2}) - P(\xi ,{u_1})} \|}_Y} \le C\| \xi \|_X^{n + \alpha },}&{\xi : = {u_2} - {u_1},}
			\end{array}
		\end{equation}
		where $C > 0$ and $\alpha \in (0,1)$. Then, $F \in {C^n}(U,Y)$, $P(\xi ,u) = J_\xi ^nF(u)$, for all $u \in U$, and ${F^{( n )}}( u )$ is locally H\"{o}lder continuous on $U$ with exponent $\alpha$. Furthermore, if $U$ is convex, then $F \in {C^{n,\alpha}}( {U,Y} )$ and
		\begin{equation}
			{\| {{F^{(n)}}({u_2}) - {F^{(n)}}({u_1})} \|_{{{\mathcal L}_s}( {{X^n},Y} )}} \le C\| {{u_2} - {u_1}} \|_X^\alpha ,
		\end{equation}
		where $C$ depends only on $n$ and $\alpha$.
	\end{theorem}
	To obtain the proof of these classical results, see e.g., \cite{HJ2014}. The following theorem states that given a function defined on a closed subset of a finite-dimensional space, it is possible to extend this function smoothly to the entire space while preserving certain properties of the original function.
	\begin{theorem}[Whitney Extension Theorem]\label{Whitney Extension Theorem}
		Let dim $X < \infty$ and $V$ be a subset of $X$. Assume also that we are given a family of polynomials $\{ P(\xi ,u),u \in V\} \subset {{\mathcal P}^n}(X,Y)$ which satisfies the following compatibility condition with some $\alpha \in (0,1)$,
		\begin{equation}\label{compatibility condition}
			{\| {P(\xi + \eta ,{u_1}) - P(\xi ,{u_2})} \|_Y} \le C{({\| \eta \|_X} + {\| \xi \|_X})^{n + \alpha }}
		\end{equation}
		for all ${u_1},{u_2} \in V$, $\eta : = {u_2} - {u_1}$ and ${\xi \in X}$. Then, there exists a function $F \in {C^{n,\alpha }}(X,Y)$ such that $J_\xi ^nF( u ) = P( {\xi ,u} )$ for all $u \in V$.
	\end{theorem}
	For the proof of this Theorem, see e.g., \cite{FC2005}. It serves as the main tool to establish the smooth extension in our subsequent analysis. To simplify notation and avoid ambiguity, we reformulate the inequality \eqref{Converse Taylor Theorem remainder term} and the compatibility condition \eqref{compatibility condition} as follows:
	\begin{equation*}
		\begin{array}{*{20}{c}}
			{F({u_2}) - P(\xi ,{u_1}) = {O^Y}(\| \xi \|_X^{n + \alpha }),}&{P(\xi + \eta ,{u_1}) - P(\xi ,{u_2}) = {O^Y}({{({{\| \eta \|}_X} + {{\| \xi \|}_X})}^{n + \alpha }}).}
		\end{array}
	\end{equation*}
	
	\section{The construction of IM}\label{Section 3}
	\noindent
	\par
	The aim of this section is to discuss the construction of an IM for the one-dimensional Burgers equation based on the smooth auxiliary transformation. Let us consider the following one-dimensional Burgers equation:
	\begin{equation*}
		\begin{array}{*{20}{c}}
			{{\partial _t}u - \nu {\partial _{xx}}u = u{\partial _x}u + g(x),}&{x \in \Omega = (0,\pi ),}&{u{|_{x = 0}} = u{|_{x = \pi }} = 0,}
		\end{array}
	\end{equation*}
	where $g \in H_0^1(\Omega )$ and the kinematic viscosity $\nu=1$. It is well-known that for any initial value ${u_0} \in H_0^1$, the equation \eqref{main equation} is globally well-posed in the class of solutions $u \in C([0,T],H_0^1) \cap {L^2}([0,T],{H^2})$ for all $T > 0$ (see e.g., \cite{BPZ2014, KZ2017, TR1997}). Therefore, this equation generates the semigroup in $H_0^1$ via
	\begin{equation*}
		\begin{array}{*{20}{c}}
			{S(t):H_0^1 \to H_0^1,}&{S(t){u_0}: = u(t),}&{t \ge 0.}
		\end{array}
	\end{equation*}
	Furthermore, for a sufficiently large constant $R$, the solution semigroup $S(t)$ admits an absorbing ball given by $\mathcal{B}: = \{ x \in H_0^1,{\| x \|_{H_0^1}} \le R\} $. Consequently, the semigroup $S(t)$ possesses a global attractor $\mathcal{A}$ in the phase space $H_0^1$. Since any trajectory will enter the absorbing ball within a finite time, the solutions outside the absorbing ball are not significant for analyzing the long-term behavior of the system. Based on this, we can neglect the influence of nonlinear term outside the absorbing ball.
	\par
	\subsection{Smooth auxiliary transformation}
	\noindent
	\par
	In this section, we briefly introduce transformation methods that are widely used in mathematical and physical models (see, e.g., \cite{SY1992, V2008, V2009} and the references therein). In 2009, J. Vukadinovic pioneered the use of nonlocal and Cole-Hopf transformations, successfully constructing IMs for the Smoluchowski equation on the circle and sphere, and further establishing an IM for the diffusive Burgers equation in the low-wavenumber (see, e.g., \cite{V2008, V2009, V2011} for details). However, the Cole-Hopf transformation is merely a Lipschitz diffeomorphism, and thus can only be used to construct Lipschitz-smooth IMs/IFs. In 2017, A. Kostianko and S. Zelik \cite{KZ2017, KZ2018} developed a smooth auxiliary transformation and constructed the ${C^{1,\varepsilon }}$-smooth IM/IF for reaction-diffusion-advection systems. 
	\par
	Following this idea, we transform the Burgers equation \eqref{main equation} into a new form via the smooth auxiliary transformation. To construct smooth IMs/IFs, we further establish the boundedness of the higher-order Fr\'{e}chet derivatives of the nonlinear term after this transformation.
	\par
	Consider the transformation $u(x,t) = a(x,t)v(x,t)$, where $a(x,t) \in \mathbb{R}\backslash \{ 0\} $. Substituting into equation \eqref{main equation} yields
	\begin{equation}
		{\partial _t}v - \partial _x^2v = {a^{ - 1}}[ava + 2{\partial _x}a]{\partial _x}v + [{a^{ - 1}}\partial _x^2a - {a^{ - 1}}{\partial _t}a + {a^{ - 1}}av{\partial _x}a]v + {a^{ - 1}}g(x).
	\end{equation}
	We impose the following structural condition on $a = a(v)$:
	\begin{equation}\label{The condition of a(v)}
		\begin{array}{*{20}{c}}
			{\frac{d}{{dx}}a = - \frac{1}{2}{P_K}(av)a,}&{a{|_{x = 0}} = 1,}
		\end{array}
	\end{equation}
	where $P_K$ is the orthoprojector onto the first $K$ eigenvectors of the Dirichlet Laplacian $- {\partial _{xx}}$ on $\Omega $. Moreover, ${a^{ - 1}}(v)$ solves the equation
	\begin{equation*}
		\begin{array}{*{20}{c}}
			{\frac{d}{{dx}}{a^{ - 1}} = \frac{1}{2}{P_K}(av){a^{ - 1}},}&{{{ {{a^{ - 1}}} |}_{x = 0}} = 1.}
		\end{array}
	\end{equation*}
	Then the following lemma holds.
	\par
	\begin{lemma}\label{Zelik lemma 1}
		For any given $R > 0$, $\mathbf{m} \in {\mathbb{N}^ + }$ and $v \in H^1(0,\pi)$, there exists a constant ${K_0}(R)$ such that for any $K \ge {K_0}(R)$, the following results hold:
		\begin{enumerate}[label=\textbf{\roman{enumi}}., leftmargin=0.75cm]
			\item Problem \ref{The condition of a(v)} admits a unique solution $a \in W^{1,\infty}(0,\pi;\mathbb{R} \backslash \{ 0\} )$ with
			\begin{equation*}
				{\| a \|_{{W^{1,\infty }}}} + {\| {{a^{ - 1}}} \|_{{W^{1,\infty }}}} \le C,
			\end{equation*}
			where the constant $C$ is independent of $K$ and $v$.
			\item The maps $v \mapsto a(v)$ and $v \mapsto {a^{ - 1}}(v)$ belong to $C^\infty ({B_{{H^1}}}(0,R),{W^{1,\infty }}(0,\pi ;\mathbb{R} \backslash \{ 0\} ))$, with all derivatives bounded $($i.e., in $C_b^\mathbf{m}$$)$. In particular, the following estimate holds:
			\begin{equation*}
				\begin{array}{*{20}{c}}
					{{{\| {a({v_1}) - a({v_2})} \|}_{{W^{1,\infty }}}} + {{\| {{a^{ - 1}}({v_1}) - {a^{ - 1}}({v_2})} \|}_{{W^{1,\infty }}}} \le C{{\| {{v_1} - {v_2}} \|}_{{H^1}}},}&{v_1},{v_2} \in {B_{{H^1}}}(0,R),
				\end{array}
			\end{equation*}
			where ${B_{{H^1}}}(0,R)$ denotes the $H^1$-ball of radius $R$ and the constant $C$ depends on $R$, but is independent of $K$.
			\item The transformation $u = \mathcal{U}(v) := a(v)v$ defines a ${C^\infty }$ map ${\mathcal U}:{B_{{H^1}}}(0,R) \to {H^1}$, with all derivatives bounded $($i.e., in $C_b^\mathbf{m}$$)$ and satisfying
			\begin{equation*}
				\begin{array}{*{20}{c}}
					{{{\| {\mathcal{U}({v_1}) - \mathcal{U}({v_2})} \|}_{{H^1}}} \le C{{\| {{v_1} - {v_2}} \|}_{{H^1}}},}&{v_1},{v_2} \in {B_{{H^1}}}(0,R),
				\end{array}
			\end{equation*}
			where the constant $C$ depends on $R$, but is independent of $K$ and $v_i$. The bounds on derivatives of all orders are independent of $K$.
		\end{enumerate}
	\end{lemma}
	\par
	The inverse map $v = \mathcal{V}(u)$ is defined as $\mathcal{V}(u) = {b^{ - 1}}(u)u$, where ${b^{ - 1}} = {b^{ - 1}}(u)$ solves the linear ODE:
	\begin{equation*}
		\begin{array}{*{20}{c}}
			{\frac{d}{{dx}}{b^{ - 1}} = - \frac{1}{2}{P_K}(u){b^{ - 1}},}&{{{ {{b^{ - 1}}} |}_{x = 0}} = 1.}
		\end{array}
	\end{equation*}
	Moreover, $b = b(u)$ solves the equation
	\begin{equation}\label{The condition of b(u)}
		\begin{array}{*{20}{c}}
			{\frac{d}{{dx}}b = \frac{1}{2}{P_K}(u)b,}&{{{ b |}_{x = 0}} = 1.}
		\end{array}
	\end{equation}
	Since the above problem is linear, we can explicitly solve it and obtain
	\begin{equation}\label{Estimate the key to pta}
		b(u)(x) = {e^{\frac{1}{2}\int_0^x {{P_K}u(s)ds} }}.
	\end{equation}
	Therefore, we have the following Lemma.
	\begin{lemma}\label{Zelik lemma 2}
		For any given $K, \mathbf{m} \in {\mathbb{N}^ + }$, $R > 0$ and $u \in H^1(0,\pi)$, the following results hold:
		\begin{enumerate}[label=\textbf{\roman{enumi}}., leftmargin=0.75cm]
			\item Problem \ref{The condition of b(u)} admits a unique solution $b \in {W^{1,\infty }}(0,\pi ;\mathbb{R} \backslash \{ 0\} )$ with
			\begin{equation*}
				{\| b \|_{{W^{1,\infty }}}} + {\| {{b^{ - 1}}} \|_{{W^{1,\infty }}}} \le C,
			\end{equation*}
			where the constant $C$ is independent of $K$ and $u$.
			\item The maps $u \mapsto {b^{ - 1}}(u)$ and $u \mapsto b(u)$ belong to ${C^\infty }({H^1}(0,\pi ),{W^{1,\infty }}(0,\pi ;\mathbb{R} \backslash \{ 0\} ))$, with all derivatives bounded $($i.e., in $C_b^\mathbf{m}$$)$. In particular, the following estimate holds:
			\begin{equation*}
				\begin{array}{*{20}{c}}
					{{{\| {b({u_1}) - b({u_2})} \|}_{{W^{1,\infty }}}} + {{\| {{b^{ - 1}}({u_1}) - {b^{ - 1}}({u_2})} \|}_{{W^{1,\infty }}}} \le C{{\| {{u_1} - {u_2}} \|}_{{H^1}}},}&{u_1},{u_2} \in {B_{{H^1}}}(0,R),
				\end{array}
			\end{equation*}
			where the constant $C$ is independent of $K$.
			\item The inverse transformation $v = \mathcal{V}(u) := b^{-1}(u)u$ defines a ${C^\infty }$ map ${\mathcal V}:{B_{{H^1}}}(0,R) \to {H^1}$, with all derivatives bounded $($i.e., in $C_b^\mathbf{m}$$)$ and satisfying
			\begin{equation*}
				\begin{array}{*{20}{c}}
					{{{\| {{\mathcal V}({u_1}) - {\mathcal V}({u_2})} \|}_{{H^1}}} \le C{{\| {{u_1} - {u_2}} \|}_{{H^1}}},}&{u_1},{u_2} \in {B_{{H^1}}}(0,R),
				\end{array}
			\end{equation*}
			where constant $C$ is independent of $K$ and $u_i$. The bounds on the derivatives of all orders are independent of $K$.
		\end{enumerate}
	\end{lemma}
	For the proofs of the above two lemmas and more details, see \cite{KZ2017, KZ2018}.
	\par
	According to Lemmas \ref{Zelik lemma 1} and \ref{Zelik lemma 2}, for any integer $K \ge {K_0}(CR)$, the maps $\mathcal{V}$ and ${\mathcal{V}^{ - 1}} = \mathcal{U}$ establish a $C^\infty$-diffeomorphism between ${B_{{H^1}}}(0,R)$ and ${\mathcal V}({B_{{H^1}}}(0,R)) \subset {B_{{H^1}}}(0,CR) \subset {H^1}$, where the constant $C$ is independent of $K$. Note that the maps $\mathcal{U}$ and $\mathcal{V}$ act not only from $H^1(0,\pi)$ to $H^1(0,\pi)$, but also from $H^1_0(0,\pi)$ to $H^1_0(0,\pi)$.
	\par
	Our next task is to transform equation \eqref{main equation} into the analogous equation with respect to the new dependent variable $v=\mathcal{V}(u)$. To this end, we fix the radius $R$ in such a way that the set ${\mathcal B} = {B_{H_0^1}}(0,R)$ is an absorbing ball for the semigroup $S(t)$. Then, the maps $\mathcal{V}$ and $\mathcal{U}$ are $C^\infty$-diffeomorphisms between a neighborhood ${\mathcal B}$ of the attractor $\mathcal A$ and $\mathcal{V}({\mathcal B})$, so the change of variables $v = {\mathcal V}(u)$ is well-defined and one-to-one in this neighborhood. We now study the transformed equation, which has the form
	\begin{align}\label{Transformed Equation 1}
		{v_t} - {v_{xx}} &= [a(v)v - {P_K}(a(v)v)]{\partial _x}v + \{ [{a^{ - 1}}(v)\partial _x^2a(v) - {a^{ - 1}}(v){\partial _t}a(v) + v{\partial _x}a(v)]v + {a^{ - 1}}(v)g\} \nonumber\\
		&= :{I_1}(v) + {I_2}(v),
	\end{align}
	where the map $a$ as well as $I_i$ depends on a parameter $K$.
	\par
	In order to construct the smooth IM/IF, it is necessary to establish the boundedness of higher-order Fr\'{e}chet derivatives of $I_1$ and $I_2$. The boundedness of the first-order Fr\'{e}chet derivatives was obtained in \cite{KZ2017, KZ2018}. In this paper, we will prove that for any fixed $\mathbf{m} \in \mathbb{N}^+$, the Fr\'{e}chet derivatives $D^\mathbf{m} I_1$ and $D^\mathbf{m} I_2$ of these two nonlinear terms are bounded.
	\begin{proposition}\label{my proposition}
		For any $\mathbf{m} \in \mathbb{N}^ + $, let $\mathcal{R} > 0$ satisfy ${\mathcal V}(\mathcal{B}) \subset {B_{H_0^1}}(0, \mathcal{R})$ and $K \ge K_0(\mathcal{R})$. Then the maps ${I_1} \in C_b^{\mathbf{m}}({B_{H_0^1}}(0,{\mathcal R}),H)$ and ${I_2} \in C_b^{\mathbf{m}}({B_{H_0^1}}(0,{\mathcal R}),H_0^1)$ satisfy the following estimates:
		\begin{equation}\label{The boundary of derivatives}
			\begin{array}{*{20}{c}}
				{{{\| {I_1^{(i)}(v)} \|}_{{\mathcal L}({{(H_0^1)}^i},H)}} \le C{K^{ - 1/2}},}&{{{\| {I_2^{(i)}(v)} \|}_{{\mathcal L}({{(H_0^1)}^i},H_0^1)}} \le {C_K},}&{v \in {B_{H_0^1}}(0,{\mathcal R}),}
			\end{array}
		\end{equation}
		where ${i = 1, \cdot \cdot \cdot ,{\mathbf{m}}}$, the constant $C$ depends on $\mathbf{m}$ but is independent of $K$, while the constant $C_K$ depends on $\mathbf{m}$, $K$, and exhibits monotonic growth as $K$ increases.
	\end{proposition}
	\par
	\begin{proof}
		Firstly, we start with the operator ${I_1}(v) = [a(v)v - {P_K}(a(v)v)]{\partial _x}v$. Clearly, since $v \in H_0^1(0,\pi )$ and $a( \cdot ) \cdot \in C_b^\mathbf{m} ({B_{H_0^1}}(0,{\mathcal R}),H_0^1)$, it follows that ${I_1} \in C_b^\mathbf{m}({B_{H_0^1}}(0,\mathcal{R}),H)$. Next, we will prove that the first estimate in inequality \eqref{The boundary of derivatives} holds.
		\par
		For $v \in {B_{H_0^1}}(0,{\mathcal R})$ and $\theta \in H_0^1(0,\pi )$, the first-order Fr\'{e}chet derivative ${I'_1}(v)[\theta ]$ is:
		\begin{equation}
			{I'_1}(v)[\theta ] = (I - {P_K})(a'(v)[\theta ]v + a(v)\theta ){\partial _x}v + (I - {P_K})(a(v)v){\partial _x}\theta.
		\end{equation}
		By taking norms, we obtain:
		\begin{multline*}
			{\| {{I'_1}(v)[\theta ]} \|_H} \le {\| {(I - {P_K})(a'(v)[\theta ]v){\partial _x}v} \|_H} \\
			+ {\| {(I - {P_K})(a(v)\theta ){\partial _x}v} \|_H} + {\| {(I - {P_K})(a(v)v){\partial _x}\theta } \|_H}.
		\end{multline*}
		Using the interpolation inequality, Weyl Lemma, and the fact that $H^1$ is an algebra, we have
		\begin{align*}
			{\| {{I'_1}(v)[\theta ]} \|_H} \le& \| {(I - {P_K})(a'(v)[\theta ]v)} \|_{{L^2}}^{1/2}\| {(I - {P_K})(a'(v)[\theta ]v)} \|_{H_0^1}^{1/2}{\| v \|_{H_0^1}} \\
			&+ \| {(I - {P_K})(a(v)\theta )} \|_{{L^2}}^{1/2}\| {(I - {P_K})(a(v)\theta )} \|_{H_0^1}^{1/2}{\| v \|_{H_0^1}} \\
			&+ \| {(I - {P_K})(a(v)v)} \|_{{L^2}}^{1/2}\| {(I - {P_K})(a(v)v)} \|_{H_0^1}^{1/2}{\| \theta \|_{H_0^1}} \\
			\le& C{K^{ - 1/2}}({\| {a'(v)[\theta ]v} \|_{H_0^1}}{\| v \|_{H_0^1}} + {\| {a(v)\theta } \|_{H_0^1}}{\| v \|_{H_0^1}} + {\| {a(v)v} \|_{H_0^1}}{\| \theta \|_{H_0^1}}) \\
			\le& C{K^{ - 1/2}}{\| \theta \|_{H_0^1}},
		\end{align*}
		where the constant $C$ is independent of $K$. Therefore, for the first-order Fr\'{e}chet derivative ${I'_1}(v)$, we obtain the following estimate: 
		\begin{equation}
			{\| {{I'_1}(v)} \|_{\mathcal{L}(H_0^1,H)}} \le C{K^{ - 1/2}}.
		\end{equation}
		\par
		Note that the boundedness of higher-order derivatives can be established through an analogous procedure: At each step, the derivative operation generates finite sums composed of $P_K$-projected products, which involve the nonlinear term $a(v)$, the solution variable $v$, and their respective derivatives. By employing the interpolation inequality in conjunction with the algebraic property of the $H^1$ space, we derive bounds controlling the growth of these derivatives. Therefore, for any $\mathbf{m} \in {\mathbb{N}^ + }$, ${\| {I_1^{(\mathbf{m})}(v)} \|_{{\mathcal L}({{(H_0^1)}^\mathbf{m}},H)}} \le C{K^{ - 1/2}}$, where the constant $C$ depends on $\mathbf{m}$ but is independent of $K$.
		\par
		Next, let's analyze ${I_2}(v) = [{a^{ - 1}}(v)\partial _x^2a(v) - {a^{ - 1}}(v){\partial _t}a(v) + v{\partial _x}a(v)]v + {a^{ - 1}}(v)g(x)$. Actually, operator $I_2$ involves two nontrivial terms $\partial _x^2a(v)$ and ${\partial _t}a(v)$, the other terms can be estimated by using Lemmas \ref{Zelik lemma 1} and \ref{Zelik lemma 2}. Let us first treat $\partial _x^2a(v)$. By differentiating the Problem \eqref{The condition of a(v)} in $x$, we obtain
		\begin{equation}
			\partial _x^2a(v) = \frac{1}{4}{P_K}{(a(v)v)^2}a - \frac{1}{2}{\partial _x}({P_K}(a(v)v))a.
		\end{equation}
		Specifically, the operator $P_K$ projects $a(v)v$ onto the finite-dimensional subspace spanned by smooth eigenfunctions $\{ {e_k}\} _{k = 1}^K$, i.e.,
		\begin{equation}
			{P_K}(a(v)v) = \sum\limits_{k = 1}^K {\langle {a(v)v,{e_k}} \rangle } {e_k}.
		\end{equation}
		Consequently, this ensures that $\partial_x^2 a(v)$ is well-defined as a $C^\infty$-map from $ B_{H_0^1}(0, \mathcal{R})$ to $H_0^1(0, \pi)$.
		\par
		To establish the boundedness property for the Fr\'{e}chet derivatives of all orders of $\partial_x^2 a$, it suffices to show that the corresponding higher-order derivatives of ${\partial _x}({P_K}(a(v)v))$ are bounded. For $v \in {B_{H_0^1}}(0,{\mathcal R})$, $\theta \in H_0^1(0,\pi )$, differentiating ${\partial _x}({P_K}(a(v)v))$ gives
		\begin{equation}
			D({\partial _x}({P_K}(a(v)v)))[\theta ] = {\partial _x}{P_K}(a(v)\theta ) + {\partial _x}{P_K}(a'(v)[\theta ]v).
		\end{equation}
		Since
		\begin{multline*}
			{\| {{\partial _x}({P_K}(a(v)v))} \|_{H_0^1}} \le \sum\limits_{k = 1}^K {| {\langle {a(v)v,{e_k}} \rangle } |{{\| {{\partial _x}({e_k})} \|}_{{H^1}}}} \le \sum\limits_{k = 1}^K {{{\| {a(v)v} \|}_{{L^2}}}{{\| {{e_k}} \|}_{{L^2}}}{{\| {{\partial _x}({e_k})} \|}_{{H^1}}}} \\
			\le {\| {a(v)} \|_{{H^1}}}{\| v \|_{H_0^1}}\sum\limits_{k = 1}^K {{{\| {{\partial _x}({e_k})} \|}_{{H^1}}}} \le C\sum\limits_{k = 1}^K {k\sqrt {1 + {k^2}} } {\| {a(v)} \|_{{H^1}}}{\| v \|_{H_0^1}},
		\end{multline*}
		where constant $C$ is independent of $K$, we have
		\begin{equation}\label{The boundary of partial-x (P_K (a (v) v))}
			{\| {D({\partial _x}({P_K}(a(v)v)))[\theta ]} \|_{{H^1}}} \le {C_K}({\| {a(v)} \|_{{H^1}}}{\| \theta \|_{H_0^1}} + {\| \theta \|_{H_0^1}}{\| v \|_{H_0^1}}) \le {C_K}{\| \theta \|_{H_0^1}},
		\end{equation}
		where constant $C_K$ depends on $K$, $\mathcal{R}$, and exhibits monotonic growth as $K$ increases.
		\par
		The boundedness of higher-order derivatives for the operator ${\partial _{xx}}a(v)$ can be established through analogous methods: Each derivative introduces finite sums of terms involving $P_K$-projected products of $a(v)$, $v$, and their derivatives. The smoothness of the eigenfunctions $\{ e_k \}$, combined with the bounds on $a(v)$ and its derivatives, ensures that all such terms remain controlled. Thus, for any given $\mathbf{m} \in {\mathbb{N}^ + }$, $\partial _x^2a \in C_b^{\mathbf{m}}({B_{H_0^1}}(0,{\mathcal R}),H_0^1)$. The bounds on the derivatives of each order depend on $\mathbf{m}$, $K$, and grow with increasing $K$.
		\par
		Let us now treat the term ${\partial _t}a$. To rigorously define the term ${\partial _t}a$, we need to use the chain rule in order to find an expression for it. To do this, we derive the value ${\partial _t}{P_K}u(v)$ from equation \eqref{main equation}:
		\begin{equation}\label{Definition of partial _t{P_K}u(v)}
			{P_K}{\partial _t}u(v) = {P_K}\partial _x^2{\mathcal U}(v) + {P_K}({\mathcal U}(v){\partial _x}({\mathcal U}(v))) + {P_K}(g(x)),
		\end{equation}
		which expands explicitly into eigenfunction components:
		\begin{equation}
			{P_K}{\partial _t}u(v) = \sum\limits_{k = 1}^K {\langle {{\mathcal U}(v),\partial _x^2{e_k}} \rangle {e_k} + \sum\limits_{k = 1}^K {\langle {{\mathcal U}(v){\partial _x}{\mathcal U}(v),{e_k}} \rangle {e_k}} + \sum\limits_{k = 1}^K {\langle {g(x),{e_k}} \rangle {e_k}} } .
		\end{equation}
		Here, the map ${\mathcal U} \in C_b^\mathbf{m} ({B_{H_0^1}}(0,{\mathcal R}),H_0^1)$, $g \in H_0^1(0,\pi )$, and the eigenfunctions $\{ {e_k}\} $ are smooth. Consequently, the right-hand side is smoothly expressible in terms of $v$. It follows that the map $v \to {P_K}{\partial _t}u(v)$ is well-defined and lies in $C_b^\mathbf{m}({B_{H_0^1}}(0,{\mathcal R}),H_0^1)$ for any fixed $\mathbf{m} \in {\mathbb{N}^ + }$.
		\par
		To define ${\partial _t}a(v)(x)$, we differentiate the explicit formula \eqref{Estimate the key to pta} in time, which yields:
		\begin{equation}
			{\partial _t}a(v)(x) = \frac{1}{2}b(u)(x)\int_0^x {{P_K}{\partial _t}u(v)(s)} ds,
		\end{equation}
		where $u={\mathcal U}(v)$ and ${\partial _t}{P_K}u(v)$ is defined by the equation \eqref{Definition of partial _t{P_K}u(v)}. This shows that the map $v \to {\partial _t}a(v)$ is well-defined, and for every given positive integer $\mathbf{m}$, ${\partial _t}a \in C_b^\mathbf{m}({B_{H_0^1}}(0,{\mathcal R}),H^1)$.
		\par
		Since for any given $\mathbf{m} \in {\mathbb{N}^ + }$, all terms in ${I_2}$ belong to the $C_b^\mathbf{m}({B_{H_0^1}}(0,{\mathcal R}),{H^1})$, by applying the estimate \eqref{The boundary of partial-x (P_K (a (v) v))}, we derive that ${I_2} \in C_b^\mathbf{m}({B_{H_0^1}}(0,{\mathcal R}),H_0^1)$ and satisfies
		\begin{equation*}
			\begin{array}{*{20}{c}}
				{{{\| {I_2^{(\mathbf{m})}(v)} \|}_{{\mathcal L}({{(H_0^1)}^\mathbf{m}},H_0^1)}} \le C_K,}&{v \in {B_{H_0^1}}(0,{\mathcal R}),}
			\end{array}
		\end{equation*}
		where the constant $C_K$ depends on $\mathbf{m}$, $K$ and $\mathcal{R}$, which exhibits monotonic growth as $K$ increases.
	\end{proof}
	\par
	As shown previously, the Burgers equation \eqref{main equation} is equivalent to equation \eqref{Transformed Equation 1} at least in the neighborhood of the set $\mathcal B$. For this reason, we can instead construct an IM for equation \eqref{Transformed Equation 1}. However, the nonlinear terms $I_1$ and $I_2$ are not globally defined on $H_0^1$. To overcome this problem, we need, as usual, to cut-off the nonlinear terms outside of a ball.
	\par
	Then, we introduce a smooth cut-off function $\varphi \in C_0^\infty (\mathbb{R})$ such that
	\begin{equation}
		\left\{ {\begin{array}{*{20}{c}}
				{\varphi (z) \equiv 1,}&{z \le {r^2},}\\
				{\varphi (z) = 0,}&{z \ge {{\mathcal R}^2},}
		\end{array}} \right.
	\end{equation}
	where $r$ is such that ${\mathcal V}({\mathcal B}) \subset {B_{H_0^1}}(0,r)$, $r < {\mathcal R}$. Furthermore, for any $K\ge K_0(\mathcal{R})$ the inverse map ${\mathcal U} = {\mathcal U}(v)$ is a diffeomorphism on ${B_{H_0^1}}(0,{\mathcal R})$. Finally, we modify equation \eqref{Transformed Equation 1} as follows:
	\begin{equation}\label{Transformed Equation 2}
		{v_t} - {v_{xx}} = \varphi (\| v \|_{{H^1}}^2){I_1}(v) + \varphi (\| v \|_{{H^1}}^2){I_2}(v) = :{{{\mathcal I}_1}}(v) + {{{\mathcal I}_2}}(v).
	\end{equation}
	\par
	Utilizing the Proposition \ref{my proposition}, we obtain the following lemma.
	\begin{lemma}\label{Smoothness of the nonlinear terms}
		For any given $\mathbf{m} \in {\mathbb{N}^ + }$, $K \ge {K_0}(\mathcal R)$, the maps ${{{\mathcal I}_1}} \in C_b^\mathbf{m}(H_0^1,H)$ and ${{{\mathcal I}_2}} \in C_b^\mathbf{m}(H_0^1,H_0^1)$. In particular,
		\begin{equation*}
			\begin{array}{*{20}{c}}
				{{{\| {{{\mathcal I}'_1}(v)} \|}_{{\mathcal L}(H_0^1,H)}} \le C{K^{ - 1/2}} = :{L_1},}&{{{\| {{{\mathcal I}'_2}(v)} \|}_{{\mathcal L}(H_0^1,H_0^1)}} \le {C_K} = :{L_2},}
			\end{array}
		\end{equation*}
		where the constant $C$ depends on $\mathbf{m}$ but is independent of $K$, while the constant $C_K$ depends on $\mathbf{m}$ and $K$, and exhibits monotonic growth as $K$ increases.
	\end{lemma}
	\par
	\subsection{The construction of IM}
	\noindent
	\par
	We now construct the desired IM for equation \eqref{Transformed Equation 2}. Let us recall the definition of an IM.
	\begin{definition}\label{Definition of IM}
		The manifold ${\mathcal M} \subset H_0^1 $ is called an IM for equation \eqref{Transformed Equation 2} if the following conditions are satisfied:
		\begin{enumerate}[label=\textbf{\roman{enumi}}., leftmargin=0.75cm]
			\item The manifold $\mathcal M$ is invariant with respect to the solution semigroup $\widetilde S(t)$ of equation \eqref{Transformed Equation 2}, i.e., $\widetilde S(t){\mathcal M} = {\mathcal M}$.
			\item It can be presented as a graph of a continuous function $M:{P_N}(H_0^1) \to {Q_N}(H_0^1)$:
			\begin{equation*}
				{\mathcal M}: = \{ {v_N} + M({v_N}),{v_N} \in {P_N}(H_0^1)\}.
			\end{equation*}
			\item The exponential tracking property holds, i.e., there exist positive constants $C$ and $\alpha$ such that for every ${v_0} \in H_0^1$ and $t \ge 0$, there is ${\widetilde v_0} \in {\mathcal M}$ satisfying
			\begin{equation*}
				\begin{array}{*{20}{c}}
					{{{\| {\widetilde S(t){v_0} - \widetilde S(t){{\widetilde v}_0}} \|}_{H_0^1}} \le C{e^{ - \alpha t}}{{\| {{v_0} - {{\widetilde v}_0}} \|}_{H_0^1}},}&{t \ge 0.}
				\end{array}
			\end{equation*}
		\end{enumerate}
	\end{definition}
	\par
	We will use the Perron method to construct the IM for equation \eqref{Transformed Equation 2}. Following the Perron method, the desired manifold is found by solving the backward in time boundary value problem
	\begin{equation}\label{Perron method equation 1}
		\begin{array}{*{20}{c}}
			{{\partial _t}v - {\partial _{xx}}v = {{{\mathcal I}_1}}(v) + {{{\mathcal I}_2}}(v),}&{{P_N}v{|_{t = 0}} = p,}&{t \le 0}
		\end{array}
	\end{equation}
	in the weighted space $L_\theta ^2({\mathbb{R}_ - },H_0^1)$ with the norm $\| v \|_{L_\theta ^2({\mathbb{R}_ - },H_0^1)}^2: = \int_{ - \infty }^0 {{e^{2\theta t}}\| {v(t)} \|_{H_0^1}^2dt} $ for some properly chosen $\theta : = \frac{{{\lambda _{N + 1}} + {\lambda _N}}}{2}$. The solution of this equation is usually constructed using the Banach Contraction Theorem, and the desired map $M:{P_N}(H_0^1) \to {Q_N}(H_0^1)$ is defined via
	\begin{equation}\label{Definition of M}
		\begin{array}{*{20}{c}}
			{V(p,0) = {v_0},}&{M(p): = {Q_N}V(p,0).}
		\end{array}
	\end{equation}
	\par
	In order to solve equation \eqref{Perron method equation 1}, we will utilize the following two lemmas.
	\par
	\begin{lemma}\label{Zelik lemma 3} Assume $\theta \in ( {{\lambda _N},{\lambda _{N + 1}}} )$ and consider the equation:
		\begin{equation}\label{Perron method equation 2}
			\begin{array}{*{20}{c}}
				{{\partial _t}v - {\partial _{xx}}v = {h_1}(t) + {h_2}(t),}&{t \in \mathbb{R},}
			\end{array}
		\end{equation}
		where ${{h_1} \in L_\theta ^2(\mathbb{R},H)}$ and ${{h_2} \in L_\theta ^2(\mathbb{R},H_0^1)}$. The following results hold:
		\begin{enumerate}[label=\textbf{\roman{enumi}}., leftmargin=0.75cm]
			\item Equation \eqref{Perron method equation 2} admits a unique solution $v \in L_\theta ^2( {\mathbb{R},H_0^1 } ) \cap {C_\theta }( {\mathbb{R},H_0^1 } )$, where $v: = {{\mathcal T}_\theta } \circ (h_1+h_2)$. The linear operator ${{\mathcal T}_\theta }$ satisfies:
			\begin{equation*}
				\begin{array}{*{20}{c}}
					{{{\| {{{\mathcal T}_\theta }} \|}_{{\mathcal L}(L_\theta ^2(\mathbb{R},H),L_\theta ^2(\mathbb{R},H_0^1))}} \le \frac{{\lambda _{N + 1}^{1/2}}}{{\min \{ \theta - {\lambda _N},{\lambda _{N + 1}} - \theta \} }},}&{{{\| {{{\mathcal T}_\theta }} \|}_{{\mathcal L}(L_\theta ^2(\mathbb{R},H_0^1),L_\theta ^2(\mathbb{R},H_0^1))}} \le \frac{1}{{\min \{ \theta - {\lambda _N},{\lambda _{N + 1}} - \theta \} }}.}
				\end{array}
			\end{equation*}
			\item For the Fourier modes ${v_n}$ of the solution $v$, the following decomposition and bounds hold:
			\begin{equation}\label{Operator norm}
				\begin{array}{*{20}{c}}
					{{v_n} = {{\mathcal T}_{\theta ,n}} \circ {h_{1,n}} + {{\mathcal T}_{\theta ,n}} \circ {h_{2,n}},}\\
					{{{\| {{v_n}} \|}_{L_\theta ^2( \mathbb{R} )}} \le \frac{{\lambda _{N + 1}^{1/2}}}{{\min \{ {\theta - {\lambda _N},{\lambda _{N + 1}} - \theta } \}}}{{\| {{h_{1,n}}} \|}_{L_\theta ^2( \mathbb{R} )}} + \frac{1}{{\min \{ {\theta - {\lambda _N},{\lambda _{N + 1}} - \theta } \}}}{{\| {{h_{2,n}}} \|}_{L_\theta ^2( \mathbb{R} )}},}
				\end{array}
			\end{equation}
			where $v = \sum\nolimits_{n = 1}^\infty {{v_n}{e_n}} $, ${h_1} = \sum\nolimits_{n = 1}^\infty {{h_{1,n}}{e_n}} $, ${h_2} = \sum\nolimits_{n = 1}^\infty {{h_{2,n}}{e_n}} $ and ${{\mathcal T}_{\theta ,n}}$ denotes the solution operator for the problem ${v'_n} + ({\lambda _n} - \theta ){v_n} = {h_{1,n}} + {h_{2,n}}$ in the weighted space.
		\end{enumerate}
	\end{lemma}
	\par
	\begin{lemma}\label{Zelik lemma 4} Let $\theta \in ( {{\lambda _N},{\lambda _{N + 1}}} )$, for any $p \in {P_N}(H_0^1)$, ${h_1} \in L_\theta ^2({\mathbb{R}_ - },H)$ and ${h_2} \in L_\theta ^2({\mathbb{R}_ - },H_0^1)$, the problem
		\begin{equation}\label{Perron method equation 3}
			\begin{array}{*{20}{c}}
				{{\partial _t}v - {\partial _{xx}}v = {h_1}(t) + {h_2}(t),}&{{P_N}v{|_{t = 0}} = p,}&{t \le 0,}
			\end{array}
		\end{equation}
		possesses a unique solution $v \in L_\theta ^2({\mathbb{R}_ - },H_0^1) \cap {C_\theta }({\mathbb{R}_ - },H_0^1)$. The solution can be expressed as
		\begin{equation*}
			{v = {{\mathcal T}_\theta } \circ ({{\widetilde h}_1} + {{\widetilde h}_2}) + {\mathcal H}p},
		\end{equation*}
		where $\widetilde h(t)$ denotes the zero extension of $h$ on $\mathbb{R}_+$ and ${\mathcal H}:{P_N}(H_0^1) \to L_\theta ^2({\mathbb{R}_ - },H_0^1)$ is the solution operator for the homogeneous problem, given by ${\mathcal H}p: = \sum_{n = 1}^N {{e^{ - {\lambda _n}t}}( {p,{e_n}} ){e_n}}$. Furthermore the solution $v$ satisfies the following estimate:
		\begin{equation*}
			\begin{array}{*{20}{c}}
				{{{\| v \|}_{L_\theta ^2({\mathbb{R}_ - },H_0^1)}} \le {C_\theta }({{\| {{h_1}} \|}_{L_\theta ^2({\mathbb{R}_ - },H)}} + {{\| {{h_2}} \|}_{L_\theta ^2({\mathbb{R}_ - },H_0^1)}} + {{\| p \|}_{H_0^1}}),}\\
				{{{\| v \|}_{{C_\theta }({\mathbb{R}_ - },H_0^1)}} \le {C_\theta }({{\| {{h_1}} \|}_{L_\theta ^2({\mathbb{R}_ - },H)}} + {{\| {{h_2}} \|}_{L_\theta ^2({\mathbb{R}_ - },H_0^1)}} + {{\| p \|}_{H_0^1}}),}
			\end{array}
		\end{equation*}
		where the constant ${C_\theta }$ is independent of ${h_1}$, ${h_2}$, $p$ and the weighted space of continuous functions is defined by ${\| u \|_{{C_\theta }(\mathbb{R},H_0^1)}}: = \mathop {\max }\limits_{t \in \mathbb{R}} \{ {e^{\theta t}}{\| {u(t)} \|_{H_0^1}}\} $.
	\end{lemma}
	The proofs of the above two lemmas are based on the decomposition of the solution $v(t)$ with respect to the basis $\{ {e_n}\} _{n = 1}^\infty $ and solving the corresponding ODEs (see e.g., \cite{KZ2017, KZ2018, KZ2024, ZS2014}).
	\par
	Now let's return to problem \eqref{Perron method equation 1}. We employ Lemma \ref{Zelik lemma 4} to reformulate it as a fixed point problem in the space $L_\theta^2(\mathbb{R}_-,H_0^1)$:
	\begin{equation}\label{Fixed point problem 1}
		v = {{\mathcal T}_\theta } \circ ({\widetilde {{\mathcal I}_1}}(v) + {\widetilde {{\mathcal I}_2}}(v)) + {\mathcal H}p.
	\end{equation}
	By applying the Banach Contraction Theorem, we can establish the existence of an IM for equation \eqref{Transformed Equation 2} (see, e.g., \cite{KZ2017, KZ2018} for more details). Given the equivalence between the original equation \eqref{main equation} and the transformed equation \eqref{Transformed Equation 1} within the neighborhood of $\mathcal{B}$, this result consequently demonstrates the existence of an IM for the one-dimensional Burgers equation. The following theorem establishes the existence of the IM.
	\par
	\begin{theorem}\label{Anna Zelik Theorem}
		Equation \eqref{main equation} possesses an IM in the phase space $H_0^1(0,\pi )$.
	\end{theorem}
	
	\section{The construction of the smooth extension of IM}\label{Section 4}
	\noindent
	\par
	The primary objective of this section is to construct the ${C^{n,\varepsilon }}$-smooth extension of IM for abstract parabolic equation \eqref{abstract parabolic problem} with two nontrivial nonlinear terms via the Whitney Extension Theorem and to obtain the desired ${C^{n,\varepsilon }}$-smooth IF by using spectral projection.
	\par
	First, we need to construct a family of IMs. Using the Perron method, the desired manifolds for equation \eqref{abstract parabolic problem} are obtained by solving the backward in time boundary value problems:
	\begin{equation}\label{my Perron method equation}
		\begin{array}{*{20}{c}}
			{{\partial _t}v + Av = {F_1}(v) + {F_2}(v),}&{{P_{{N_i}}}v{|_{t = 0}} = p,}&{t \le 0}
		\end{array}
	\end{equation}
	in the weighted space $L_{{\theta _i}}^2({\mathbb{R}_ - },H_0^1)$ with the norm $\| v \|_{L_{{\theta _i}}^2({\mathbb{R}_ - },H_0^1)}^2: = \int_{ - \infty }^0 {{e^{2{\theta _i}t}}\| {v(t)} \|_{H_0^1}^2dt} $, where the weight parameter $\theta _i$ is chosen as
	\begin{equation}
		{\theta _i} = {\lambda _{{N_i}}} + \gamma + \lambda _{{N_i} + 1}^{1/2}{L_1} + {L_2},
	\end{equation}
	with
	\begin{equation}
		\gamma = \mathop {\min }\limits_{i \in \{ 1, \cdot \cdot \cdot ,n\} } \{ {\frac{{{\lambda _{{N_i} + 1}} - {\lambda _{{N_i}}} - (i - 1){\lambda _{{N_{i - 1}}}} - \lambda _{{N_i} + 1}^{1/2}(i + 1){L_1} + (i + 1){L_2}}}{{i + 1}}} \}.
	\end{equation}
	To reduce the complexity of expression, we will no longer distinguish between the $p$ obtained through different projection methods. Under this choice of $\theta _i$ and conditions \eqref{my spectral gap condition}, it follows that ${\theta _i} \in ({\lambda _{{N_i}}},{\lambda _{{N_{i + 1}}}})$. The solution of this equation is constructed using the Banach Contraction Theorem, and the desired maps ${M_{{N_i}}}:{P_{{N_i}}}(H_0^1) \to {Q_{{N_i}}}(H_0^1)$ are defined via $V(p,0) = {v_0}$, ${M_{{N_i}}}(p): = {Q_{{N_i}}}V(p,0)$.
	\par
	By Lemma \ref{Zelik lemma 4}, we reformulate the above equations as the fixed point problems in the space $L_{{\theta _i}}^2({\mathbb{R}_ - },H_0^1)$:
	\begin{equation}
		\begin{array}{*{20}{c}}
			{v = {{\mathcal T}_{{\theta _i}}} \circ ({\widetilde F_1}(v) + {\widetilde F_2}(v)) + {{\mathcal H}_i}p,}&{i = 1,2, \cdot \cdot \cdot ,n,}
		\end{array}
	\end{equation}
	where ${{\mathcal H}_i}p: = \sum\nolimits_{n' = 1}^{{N_i}} {{e^{ - {\lambda _{n'}}t}}(p,{e_{n'}}){e_{n'}}} $. Thus, the Lipschitz constant $Lip$ of the right-hand side of the above equations can be estimated by
	\begin{multline*}
		Lip \le \frac{{\lambda _{{N_i} + 1}^{1/2}{L_1}}}{{\min \{ {{\theta _i} - {\lambda _{{N_i}}},{\lambda _{{N_i} + 1}} - {\theta _i}} \}}} + \frac{{{L_2}}}{{\min \{ {{\theta _i} - {\lambda _{{N_i}}},{\lambda _{{N_i} + 1}} - {\theta _i}} \}}} \le \frac{{\lambda _{{N_i} + 1}^{1/2}{L_1} + {L_2}}}{{\gamma + \lambda _{{N_i} + 1}^{1/2}{L_1} + {L_2}}} < 1.
	\end{multline*}
	Hence, by applying the Banach Contraction Theorem, there exists a family of functions $\{ {M_{{N_i}}}:{P_{{N_i}}}(H_0^1) \to {Q_{{N_i}}}(H_0^1),i = 1, \cdot \cdot \cdot ,n\} $ and the corresponding IMs $\{ {{\mathcal M}_{{N_i}}}\} $ that satisfy the following nested relationship:
	\begin{equation*}
		{{\mathcal M}_{{N_1}}} \subset {{\mathcal M}_{{N_2}}} \subset \cdot \cdot \cdot \subset {{\mathcal M}_{{N_n}}},
	\end{equation*}
	where ${N_i}$ is the dimension of the $i$th IM and the position of the associated eigenvalue in the spectrum is given by ${\lambda _{{N_i}}}$. In the subsequent analysis, we focus on the smoothness of these IMs. The following lemma is crucial for our analysis.
	\begin{lemma}\label{my main lemma}
		Under the assumptions of Theorem \ref{my theorem} hold, the following hold for all $i = 1, \cdot \cdot \cdot ,n$ and $j = 0, \cdot \cdot \cdot ,i - 1$.
		\begin{enumerate}[label=\textbf{\roman{enumi}}., leftmargin=0.75cm]
			\item There exists a family of exponents ${\theta _i} + j{\theta _{i - 1}}$ satisfying:
			\begin{equation}\label{control spectrum 2}
				\left\{ {\begin{array}{*{20}{c}}
						{{\theta _i} + j{\theta _{i - 1}} \in ({\lambda _{{N_i}}},{\lambda _{{N_i} + 1}}),}\\
						{\frac{{\lambda _{{N_i} + 1}^{1/2}{L_1}}}{{\min \{ {\theta _i} + j{\theta _{i - 1}} - {\lambda _{{N_i}}},{\lambda _{{N_i} + 1}} - {\theta _i} - j{\theta _{i - 1}}\} }} + \frac{{{L_2}}}{{\min \{ {\theta _i} + j{\theta _{i - 1}} - {\lambda _{{N_i}}},{\lambda _{{N_i} + 1}} - {\theta _i} - j{\theta _{i - 1}}\} }} < 1.}
				\end{array}} \right.
			\end{equation}
			\item For arbitrary function $v \in C({\mathbb{R}_ - },H_0^1)$, $p \in {(H_0^1)_{{N_i}}}$, ${h_1} \in L_{{\theta _i} + j{\theta _{i - 1}} + \varepsilon }^2({\mathbb{R}_ - },H)$ and ${h_2} \in L_{{\theta _i} + j{\theta _{i - 1}} + \varepsilon }^2({\mathbb{R}_ - },H_0^1)$, the corresponding equation of variations
			\begin{equation}\label{Perron method equation 5}
				\left\{ {\begin{array}{*{20}{c}}
						{{\partial _t}w + Aw - ({F'_1}(v)w + {F'_2}(v)w) = {h_1}(t) + {h_2}(t),}\\
						{\begin{array}{*{20}{c}}
								{{P_{{N_i}}}w{|_{t = 0}} = p,}&{t \le 0.}
						\end{array}}
				\end{array}} \right.
			\end{equation}
			admits a unique solution $w \in L_{{\theta _i} + j{\theta _{i - 1}}}^2({\mathbb{R}_ - },H_0^1 ) \cap {C_{{\theta _i} + j{\theta _{i - 1}}}}({\mathbb{R}_ - },H_0^1 )$ with estimates:
			\begin{equation*}
				\begin{array}{*{20}{c}}
					{{{\| w \|}_{L_{{\theta _i} + j{\theta _{i - 1}}}^2({\mathbb{R}_ - },H_0^1)}} \le {C_{L,{\theta _i} + j{\theta _{i - 1}}}}({{\| {{h_1}} \|}_{L_{{\theta _i} + j{\theta _{i - 1}}}^2({\mathbb{R}_ - },H)}} + {{\| {{h_2}} \|}_{L_{{\theta _i} + j{\theta _{i - 1}}}^2({\mathbb{R}_ - },H_0^1)}} + {{\| {p} \|}_{H_0^1}}),}\\
					{{{\| w \|}_{{C_{{\theta _i} + j{\theta _{i - 1}}}}({\mathbb{R}_ - },H_0^1)}} \le {C_{L,{\theta _i} + j{\theta _{i - 1}}}}({{\| {{h_1}} \|}_{L_{{\theta _i} + j{\theta _{i - 1}}}^2({\mathbb{R}_ - },H)}} + {{\| {{h_2}} \|}_{L_{{\theta _i} + j{\theta _{i - 1}}}^2({\mathbb{R}_ - },H_0^1)}} + {{\| {p} \|}_{H_0^1}}),}
				\end{array}
			\end{equation*}
			where ${F'_1}(v) \in {\mathcal L}(H_0^1,H)$, ${F'_2}(v) \in {\mathcal L}(H_0^1,H_0^1)$ and the constant ${C_{L,{\theta _i} + j{\theta _{i - 1}}}}$ is independent of $h$ and $p$.
		\end{enumerate}
	\end{lemma}
	\par
	\begin{proof}
		We first prove part (i) of Lemma \ref{my main lemma}. Adopting the previous assumptions, we define
		\begin{equation}
			\begin{array}{*{20}{c}}
				{{\theta _i} = {\lambda _{{N_i}}} + \gamma + \lambda _{{N_i} + 1}^{1/2}{L_1} + {L_2},}&{{\theta _0} = 0,}&{i = 1, \cdot \cdot \cdot ,n,}
			\end{array}
		\end{equation}
		where
		\begin{equation}
			\gamma = \mathop {\min }\limits_{i \in \{ 1, \cdot \cdot \cdot ,n\} } \{ {\frac{{{\lambda _{{N_i} + 1}} - {\lambda _{{N_i}}} - (i - 1){\lambda _{{N_{i - 1}}}} - \lambda _{{N_i} + 1}^{1/2}(i + 1){L_1} + (i + 1){L_2}}}{{i + 1}}} \}.
		\end{equation}
		Under conditions \eqref{my spectral gap condition}, for all $i = 1, \cdot \cdot \cdot ,n$ and $j = 0, \cdot \cdot \cdot ,i - 1$, we obtain the following estimate:
		\begin{multline*}
			{\lambda _{{N_i}}} + \gamma + \lambda _{{N_i} + 1}^{1/2}{L_1} + {L_2} \le {\theta _i} + (j - 1){\theta _{i - 1}} \le {\theta _i} + (i - 1){\theta _{i - 1}} \\
			= {\lambda _{{N_i}}} + \gamma + \lambda _{{N_i} + 1}^{1/2}{L_1} + {L_2} + (i - 1)({\lambda _{{N_{i - 1}}}} + \gamma + \lambda _{{N_{i - 1}} + 1}^{1/2}{L_1} + {L_2}) \\
			< {\lambda _{{N_i}}} + (i - 1)({\lambda _{{N_{i - 1}}}}) + i(\gamma + \lambda _{{N_i} + 1}^{1/2}{L_1} + {L_2}) < {\lambda _{{N_{i + 1}}}} - (\lambda _{{N_i} + 1}^{1/2}{L_1} + {L_2}),
		\end{multline*}
		where we used that ${\lambda _{{N_{i - 1}}}} < {\lambda _{{N_{i - 1}} + 1}} \le {\lambda _{{N_i}}} < {\lambda _{{N_i} + 1}}$. This yields the following estimate:
		\begin{multline*}
			\frac{1}{{\min \{ {\theta _i} + j{\theta _{i - 1}} - {\lambda _{{N_i}}},{\lambda _{{N_i} + 1}} - {\theta _i} - j{\theta _{i - 1}}\} }} \\
			\le \frac{1}{{\min \{ {\theta _i} - {\lambda _{{N_i}}},{\lambda _{{N_i} + 1}} - {\theta _i} - (i - 1){\theta _{i - 1}}\} }} \le \frac{1}{{\gamma + \lambda _{{N_i} + 1}^{1/2}{L_1} + {L_2}}}.
		\end{multline*}
		Consequently, the exponents ${\theta _i} + j{\theta _{i - 1}}$ satisfy the bound:
		\begin{multline*}
			\frac{{\lambda _{{N_i} + 1}^{1/2}{L_1}}}{{\min \{ {\theta _i} + j{\theta _{i - 1}} - {\lambda _{{N_i}}},{\lambda _{{N_i} + 1}} - {\theta _i} - j{\theta _{i - 1}}\} }} + \frac{{{L_2}}}{{\min \{ {\theta _i} + j{\theta _{i - 1}} - {\lambda _{{N_i}}},{\lambda _{{N_i} + 1}} - {\theta _i} - j{\theta _{i - 1}}\} }} \\
			\le \frac{{\lambda _{{N_i} + 1}^{1/2}{L_1} + {L_2}}}{{\gamma + \lambda _{{N_i} + 1}^{1/2}{L_1} + {L_2}}} < 1.
		\end{multline*}
		\par
		Next, we will prove the second part of this lemma.
		\par
		As shown above, the problem \eqref{Perron method equation 5} can be solved via the Banach Contraction Theorem treating the term ${F'_1}(v)w + {F'_2}(v)w$ as a perturbation analogously to the nonlinear case. We can write the problem \eqref{Perron method equation 5} as a fixed point problem
		\begin{equation}\label{Fixed point problem 3}
			w = {{\mathcal T}_{{\theta _i} + j{\theta _{i - 1}}}} \circ ({\widetilde F_1}^\prime (v)w + {\widetilde F_2}^\prime (v)w) + {{\mathcal T}_{{\theta _i} + j{\theta _{i - 1}}}} \circ (\widetilde {{h_1}}(t) + \widetilde {{h_2}}(t)) + {{\mathcal H}_i}p
		\end{equation}
		in the space $L_{{\theta _i} + j{\theta _{i - 1}}}^2( {{\mathbb{R}_ - },H} )$, where ${{\mathcal H}_i}p: = \sum\nolimits_{n' = 1}^{{N_i}} {{e^{ - {\lambda _{n'}}t}}(p,{e_{n'}}){e_{n'}}} $.
		\par
		To achieve this, it is sufficient to verify that the norm of ${{\mathcal T}_{{\theta _i} + j{\theta _{i - 1}}}} \circ ({\widetilde F_1}^\prime (v) + {\widetilde F_2}^\prime (v))$ is less than one. Hence, by utilizing the estimate \eqref{Operator norm}, the norm of ${{\mathcal T}_{{\theta _i} + j{\theta _{i - 1}}}} \circ ({\widetilde F_1}^\prime (v) + {\widetilde F_2}^\prime (v))$ does not exceed ${\frac{{\lambda _{{N_i} + 1}^{1/2}{L_1}}}{{\min \{ {\theta _i} + j{\theta _{i - 1}} - {\lambda _{{N_i}}},{\lambda _{{N_i} + 1}} - {\theta _i} - j{\theta _{i - 1}}\} }} + \frac{{{L_2}}}{{\min \{ {\theta _i} + j{\theta _{i - 1}} - {\lambda _{{N_i}}},{\lambda _{{N_i} + 1}} - {\theta _i} - j{\theta _{i - 1}}\} }}}$.
		\par
		Following the aforementioned selection, we obtain
		\begin{equation*}
			\frac{{\lambda _{{N_i} + 1}^{1/2}{L_1}}}{{\min \{ {\theta _i} + j{\theta _{i - 1}} - {\lambda _{{N_i}}},{\lambda _{{N_i} + 1}} - {\theta _i} - j{\theta _{i - 1}}\} }} + \frac{{{L_2}}}{{\min \{ {\theta _i} - {\lambda _{{N_i}}},{\lambda _{{N_i} + 1}} - {\theta _i} - (i - 1){\theta _{i - 1}}\} }} < 1.
		\end{equation*}
		\par
		Thus, there exists a unique solution to the problem \eqref{Perron method equation 5}. Similarly, we have
		\begin{equation*}
			\begin{array}{*{20}{c}}
				{{{\| w \|}_{L_{{\theta _i} + j{\theta _{i - 1}}}^2({\mathbb{R}_ - },H_0^1)}} \le {C_{L,{\theta _i} + j{\theta _{i - 1}}}}({{\| {{h_1}} \|}_{L_{{\theta _i} + j{\theta _{i - 1}}}^2({\mathbb{R}_ - },H)}} + {{\| {{h_2}} \|}_{L_{{\theta _i} + j{\theta _{i - 1}}}^2({\mathbb{R}_ - },H_0^1)}} + {{\| {p} \|}_{H_0^1}}),}\\
				{{{\| w \|}_{{C_{{\theta _i} + j{\theta _{i - 1}}}}({\mathbb{R}_ - },H_0^1)}} \le {C_{L,{\theta _i} + j{\theta _{i - 1}}}}({{\| {{h_1}} \|}_{L_{{\theta _i} + j{\theta _{i - 1}}}^2({\mathbb{R}_ - },H)}} + {{\| {{h_2}} \|}_{L_{{\theta _i} + j{\theta _{i - 1}}}^2({\mathbb{R}_ - },H_0^1)}} + {{\| {p} \|}_{H_0^1}}),}
			\end{array}
		\end{equation*}
		where the constant ${C_{L,{\theta _i} + j{\theta _{i - 1}}}}$ are independent of $h_1$, $h_2$ and $p$.
	\end{proof}
	\par
	In order to construct the $C^{n,{\varepsilon} }$-smooth extension of IM, we first need to prove that each IM in the family $\{ {{{\mathcal M}_{{N_1}}},{{\mathcal M}_{{N_2}}}, \cdot \cdot \cdot ,{{\mathcal M}_{{N_n}}}} \}$ is $C^{1,{\varepsilon}}$-smooth. Hence, we will utilize Lemma \ref{my main lemma} to prove that the above obtained family of IMs is ${C^{1,\varepsilon }}$-smooth.
	\par
	\begin{proposition}\label{C1 smooth IM}
		Under the assumptions of Theorem \ref{my theorem}, the following statements hold:
		\begin{enumerate}[label=\textbf{\roman{enumi}}., leftmargin=0.75cm]
			\item The function $V$ defined in \eqref{my Perron method equation} is ${C^{1,\varepsilon }}$-smooth, so the associated IM ${\{ {{\mathcal M}_{{N_i}}}\} _{i = 1, \cdot \cdot \cdot ,n}}$ is ${C^{1,\varepsilon }}$-smooth.
			\item For any $p,\xi \in {P_{{N_i}}}(H_0^1)$, the first order Fr\'{e}chet derivative ${M'_{{N_i}}}(p)\xi $ corresponds to the value of ${Q_{{N_i}}}$ projection of the ${V'_\xi }: = V'( {p,t} )\xi $ at $t = 0$, where the function ${V'_\xi }$ solves the equation of variations:
			\begin{equation}\label{C1 smooth equation}
				\left\{ {\begin{array}{*{20}{c}}
						{{\partial _t}{V'_\xi } + A{V'_\xi } - ({F'_1}(v){V'_\xi } + {F'_2}(v){V'_\xi }) = 0,}\\
						{\begin{array}{*{20}{c}}
								{{P_{{N_i}}}{V'_\xi }{|_{t = 0}} = p,}&{t \le 0.}
						\end{array}}
				\end{array}} \right.
			\end{equation}
			\item For any ${p_1},{p_2} \in {P_{{N_i}}}(H_0^1)$, the map $V$ satisfies the following estimate:
			\begin{equation}
				{\| {V({p_2},t) - V({p_1},t) - V'({p_1},t)\xi } \|_{L_{(1 + \varepsilon ){\theta _i}}^2({\mathbb{R}_ - },H_0^1)}} \le C\| {{p_2} - {p_1}} \|_{H_0^1}^{1 + \varepsilon },
			\end{equation}
			where $\xi = {p_2} - {p_1}$, ${\theta _i} = {\lambda _{{N_i}}} + \gamma + \lambda _{{N_i} + 1}^{1/2}{L_1} + {L_2}$ and $C$ is independent of ${p_1}$ and ${p_2}$.
		\end{enumerate}
	\end{proposition}
	\par
	\begin{proof}
		To prove the IM is ${C^{1,\varepsilon }}$-smooth, it suffices to show that its corresponding map ${M_{{N_i}}}$ is ${C^{1,\varepsilon }}$-smooth. Define ${p_1},{p_2} \in {P_{{N_i}}}(H_0^1)$, ${v_1}(t): = V({p_1},t)$ and ${v_2}(t): = V({p_2},t)$ be the trajectories associated with the IM. Additionally, let $\bar v(t): = {v_1}(t) - {v_2}(t)$ and $\xi : = {p_1} - {p_2}$. Then $\bar v$ satisfies the equation
		\begin{equation}\label{C0 smooth equation}
			\left\{ {\begin{array}{*{20}{c}}
					{{\partial _t}\bar v + A\bar v - (L_{{v_1},{v_2}}^1(t)\bar v + L_{{v_1},{v_2}}^2(t)\bar v) = 0,}\\
					{\begin{array}{*{20}{c}}
							{{P_{{N_i}}}\bar v{|_{t = 0}} = \xi ,}&{t \le 0,}
					\end{array}}
			\end{array}} \right.
		\end{equation}
		where operators $L_{{v_1},{v_2}}^1(t)$ and $L_{{v_1},{v_2}}^2(t)$ are defined as follows:
		\begin{equation*}
			\begin{array}{*{20}{c}}
				{L_{{v_1},{v_2}}^1(t): = \int_0^1 {{F'_1}(s{v_1}(t) + (1 - s){v_2}(t))} ds,}&{L_{{v_1},{v_2}}^2(t): = \int_0^1 {{F'_2}(s{v_1}(t) + (1 - s){v_2}(t))} ds,}
			\end{array}
		\end{equation*}
		and it holds that the norm of ${{\mathcal T}_{{\theta _i}}} \circ (L_{{v_1},{v_2}}^1(t) + L_{{v_1},{v_2}}^2(t))$ is less than one.
		\par
		Applying Lemma \ref{my main lemma} to problem \eqref{C0 smooth equation} yields
		\begin{equation}\label{Estimation of V}
			\begin{array}{*{20}{c}}
				{{{\| {\bar v} \|}_{L_{{\theta _i}}^2({\mathbb{R}_ - },H_0^1)}} \le {C_{\theta_i} }{{\| \xi \|}_{H_0^1}},}&{{{\| {\bar v} \|}_{{C_{{\theta _i}}}({\mathbb{R}_ - },H_0^1)}} \le {C_{\theta_i} }{{\| \xi \|}_{H_0^1}},}
			\end{array}
		\end{equation}
		where the constant ${C_{{\theta_i} }}$ is independent of $p_1, p_2$ and $\bar v$.
		\par
		Note that the function $V'(p,t)\xi$ is well-defined for all $p,\xi \in {P_{{N_i}}}(H_0^1)$ due to Lemma \ref{my main lemma} and satisfy the analogue of \eqref{Estimation of V}. Let $w(t): = \bar v(t) - V'({p_1},t)\xi $ with $\xi : = {p_1} - {p_2}$, which satisfies the problem:
		\begin{equation*}
			\left\{ {\begin{array}{*{20}{c}}
					{{\partial _t}w + Aw - ({F'_1}({v_1})w + {F'_2}({v_1})w) = {h_1}(t) + {h_2}(t),}\\
					{{P_{{N_i}}}w{|_{t = 0}} = 0,}
			\end{array}} \right.
		\end{equation*}
		where ${h_1}(t): = {F_1}({v_1}) - {F_1}({v_2}) - {F'_1}({v_1})\bar v$ and ${h_2}(t): = {F_2}({v_1}) - {F_2}({v_2}) - {F'_2}({v_1})\bar v$.
		\par
		According to the hypothesis ${\theta_i} = {\lambda_{N_i}} + \gamma + \lambda_{N_i+1}^{1/2}L_1 + L_2$, and given that the spectral gap condition \eqref{my spectral gap condition} satisfy the strict inequality (i.e., equality does not hold), there exists a sufficiently small $\varepsilon \in (0,1)$ such that $(1 + \varepsilon)\theta_i \in (\lambda_{N_i}, \lambda_{N_i+1})$. This condition also ensures that the Banach Contraction Theorem holds, namely:
		\begin{equation*}
			\frac{{\lambda _{{N_i} + 1}^{1/2}{L_1}}}{{\min \{ {(1 + \varepsilon ){\theta _i} - {\lambda _{{N_i}}},{\lambda _{{N_i} + 1}} - (1 + \varepsilon ){\theta _i}} \}}} + \frac{{{L_2}}}{{\min \{ {(1 + \varepsilon ){\theta _i} - {\lambda _{{N_i}}},{\lambda _{{N_i} + 1}} - (1 + \varepsilon ){\theta _i}} \}}} < 1.
		\end{equation*}
		It is particularly noteworthy that $\varepsilon$ can be chosen arbitrarily small. Therefore, even if ${\theta_i}$ is subjected to a slight perturbation and becomes $(1 + \varepsilon)\theta_i$, Lemma \ref{my main lemma} remains valid. Furthermore, since ${F_1} \in C_b^{1,\varepsilon }(H_0^1,H)$ and ${F_2} \in C_b^{1,\varepsilon }(H_0^1,H_0^1)$, we have:
		\begin{equation*}
			\begin{array}{*{20}{c}}
				{{{\| {{h_1}(t)} \|}_H} \le \| {\bar v(t)} \|_{H_0^1}^{1 + \varepsilon },}&{{{\| {{h_2}(t)} \|}_{H_0^1}} \le \| {\bar v(t)} \|_{H_0^1}^{1 + \varepsilon }.}
			\end{array}
		\end{equation*}
		This consequently establishes the estimates:
		\begin{multline*}
			\| {{h_1}(t)} \|_{L_{(1 + \varepsilon ){\theta_i} }^2({\mathbb{R}_ - },H)}^2 \le \| {\bar v} \|_{L_{(1 + \varepsilon ){\theta_i} }^2({\mathbb{R}_ - },H_0^1)}^{2(1 + \varepsilon )} = \int_{ - \infty }^0 {{e^{2{\theta_i} t}}\| {\bar v(t)} \|_{H_0^1}^2{e^{2\varepsilon {\theta_i} t}}\| {\bar v(t)} \|_{H_0^1}^{2\varepsilon }dt} \\
			\le {C_{\theta_i} }\| {\bar v} \|_{L_{\theta_i} ^2({\mathbb{R}_ - },H_0^1)}^2\| {\bar v} \|_{{C_{\theta_i} }({\mathbb{R}_ - },H_0^1)}^{2\varepsilon },
		\end{multline*}
		and
		\begin{equation*}
			\| {{h_2}(t)} \|_{L_{(1 + \varepsilon ){\theta_i} }^2({\mathbb{R}_ - },H_0^1)}^2 \le \| {\bar v} \|_{L_{(1 + \varepsilon ){\theta_i} }^2({\mathbb{R}_ - },H_0^1)}^{2(1 + \varepsilon )} \le {C_{\theta_i} }\| {\bar v} \|_{L_{\theta_i} ^2({\mathbb{R}_ - },H_0^1)}^2\| {\bar v} \|_{{C_{\theta_i} }({\mathbb{R}_ - },H_0^1)}^{2\varepsilon }.
		\end{equation*}
		\par
		Due to estimate \eqref{Estimation of V} and Lemma \ref{my main lemma}, we obtain
		\begin{multline*}
			{\| w \|_{L_{(1 + \varepsilon ){\theta_i} }^2({\mathbb{R}_ - },H_0^1)}} \le {C_{\theta_i} }({\| {{h_1}} \|_{L_{(1 + \varepsilon ){\theta_i} }^2({\mathbb{R}_ - },H)}} + {\| {{h_2}} \|_{L_{(1 + \varepsilon ){\theta_i} }^2({\mathbb{R}_ - },H_0^1)}}) \\
			\le {C_{\theta_i} }{\| {\bar v} \|_{L_{\theta_i} ^2({\mathbb{R}_ - },H_0^1)}}\| {\bar v} \|_{{C_{\theta_i} }({\mathbb{R}_ - },H_0^1)}^\varepsilon \le {C_{\theta_i} }\| \xi \|_{H_0^1}^{1 + \varepsilon },
		\end{multline*}
		and
		\begin{equation*}
			{\| w \|_{{C_{( {1 + \varepsilon } ){\theta_i} }}( {{\mathbb{R}_ - },H_0^1} )}} \le {C_{\theta_i} }\| \xi \|_{H_0^1}^{1 + \varepsilon },
		\end{equation*}
		where the constant ${C_{\theta_i} }$ is independent of $v$, ${h_1}$ and ${h_2}$. We finally derive
		\begin{equation*}
			{\| {{M_{{N_i}}}( {{p_2}} ) - {M_{{N_i}}}( {{p_1}} ) - {M'_{{N_i}}}(p)\xi } \|_{H_0^1}} = {\| {w( 0 )} \|_{H_0^1}} \le {\| w \|_{{C_{( {1 + \varepsilon } ){\theta _i}}}({\mathbb{R}_ - },H_0^1)}} \le {C_{{\theta _i}}}\| \xi \|_{H_0^1}^{1 + \varepsilon }.
		\end{equation*}
		\par
		By the Converse Taylor Theorem, it follows that the map ${{M_{{N_i}}}}$ is ${C^{1,\varepsilon }}$-smooth, and thus the family of IMs ${\{ {\mathcal{M}_{{N_i}}}\} _{i = 1, \cdot \cdot \cdot ,n}}$ constructed above is also ${C^{1,\varepsilon }}$-smooth.
	\end{proof}
	\begin{remark}
		Notably, the exponent $\varepsilon $ depends strictly on the spectral gap condition and takes a very small value. Since the primary objective of this paper is to investigate the smoothness of the IM, rather than determining the precise value of the H\"{o}lder exponent $\varepsilon $, a detailed estimation of $\varepsilon $ will not be provided in the subsequent proofs (see, e.g., \cite{KZ2024, NLS2026, ZS2014} for more details).
	\end{remark}
	\par
	\begin{remark}
		Proposition \ref{C1 smooth IM} can be naturally extended to higher-order derivatives. Specifically, the second derivative ${V''_\xi } = V''( {p,t} )[ {\xi ,\xi } ]$ satisfies
		\begin{equation*}
			\left\{ {\begin{array}{*{20}{c}}
					{{\partial _t}{V''_\xi } + A{V''_\xi } - ({F'_1}(v(t)){V''_\xi } + {F'_2}(v(t)){V''_\xi }) = {F''_1}(v(t))[{V'_\xi },{V'_\xi }] + {F''_2}(v(t))[{V'_\xi },{V'_\xi }],}\\
					{\begin{array}{*{20}{c}}
							{{P_{{N_i}}}{V''_\xi }{|_{t = 0}} = 0,}&{t \le 0.}
					\end{array}}
			\end{array}} \right.
		\end{equation*}
		It is worth noting that ${V'_\xi } \in L_{\theta_i} ^2( {{\mathbb{R}_ - },H} )$, the right-hand side ${F''_1}(v(t))[{V'_\xi },{V'_\xi }] \in L_{2{\theta_i} }^2({\mathbb{R}_ - },H)$, and ${F''_2}(v(t))[{V'_\xi },{V'_\xi }] \in L_{2{\theta_i} }^2({\mathbb{R}_ - },H_0^1)$. We require the following spectral gap condition to solve this problem:
		\begin{equation}
			\left\{ {\begin{array}{*{20}{c}}
					{2{\theta _i} \in ({\lambda _{{N_i}}},{\lambda _{{N_i} + 1}}),}\\
					{\frac{{\lambda _{{N_i} + 1}^{1/2}{L_1}}}{{\min \{ 2{\theta _i} - {\lambda _{{N_i}}},{\lambda _{N + 1}} - 2{\theta _i}\} }} + \frac{{{L_2}}}{{\min \{ 2{\theta _i} - {\lambda _{{N_i}}},{\lambda _{{N_i} + 1}} - 2{\theta _i}\} }} < 1.}
			\end{array}} \right.
		\end{equation}
		However, in practical applications, if the principal operator is a finite-order elliptic operator in a bounded domain, it is difficult to obtain such large exponential spectral gaps, which is the primary challenge in constructing the smooth IMs, see \cite{KZ2024, NLS2026} for more details.
	\end{remark}
	Based on Lemma \ref{my main lemma} and Proposition \ref{C1 smooth IM}, it is evident that for any given $n \in {\mathbb{N}^ + }$, there exists a family of ${C^{1,\varepsilon }}$-smooth functions $\{ {M_{{N_i}}}:{P_{{N_i}}}(H_0^1) \to {Q_{{N_i}}}(H_0^1),i = 1, \cdot \cdot \cdot ,n\} $ defined by ${M_{{N_i}}}(p): = {Q_{{N_i}}}V(p,0)$, and the corresponding IMs $\{ {{\mathcal M}_{{N_i}}}\} $ that satisfy the following nested relationship:
	\begin{equation*}
		{{\mathcal M}_{{N_1}}} \subset {{\mathcal M}_{{N_2}}} \subset \cdot \cdot \cdot \subset {{\mathcal M}_{{N_n}}},
	\end{equation*}
	where ${N_i}$ denotes the eigenvalue corresponding to the $i$th IM and its position in the spectrum is given by ${\lambda _{{N_i}}}$. Therefore, for any $p \in {P_{{N_n}}}{{\mathcal M}_{{N_1}}}$, we have ${P_{{N_i}}}p \in {P_{{N_i}}}{{\mathcal M}_{{N_i}}}$. Since $\dim {P_{{N_n}}}(H_0^1) < \infty $ and ${P_{{N_n}}}{{\mathcal M}_{{N_1}}} \subset {P_{{N_n}}}(H_0^1)$, we can let $X:={P_{{N_n}}}H$ and $V:= {P_{{N_n}}}{{\mathcal M}_{{N_1}}}$ in Whitney Extension Theorem \ref{Whitney Extension Theorem}.
	\par
	In the following, let us provide the definition of the smooth extension of IM (see, e.g., \cite{KZ2024, NLS2026}).
	\begin{definition}\label{Definition of IM extension}
		Let $n \in \mathbb{N}$ and for any $\mu>0$, we say that a ${C^{n,\varepsilon }}$-smooth submanifold ${\widetilde{\mathcal M}_{{N_n}}}$ of the phase space $H_0^1$ is a ${C^{n,\varepsilon }}$-smooth extension of the initial IM ${{\mathcal M}_{{N_1}}}$ for some $\varepsilon > 0$ if
		\begin{enumerate}[label=\textbf{\roman{enumi}.}, leftmargin=0.75cm] 
			\item ${{\mathcal M}_{{N_1}}} \subset {\widetilde{\mathcal M}_{{N_n}}}$.
			\par
			\item The manifold ${\widetilde{\mathcal M}_{{N_n}}}$ is $\mu $-close to the IM ${{\mathcal M}_{{N_n}}}$ in the $C_b^1$-norm, i.e.,
			\begin{equation*}
				{\| {{\widetilde{M}_{{N_n}}}( \cdot ) - {M_{{N_n}}}( \cdot )} \|_{C_b^1({P_{{N_n}}}(H_0^1),{Q_{{N_n}}}(H_0^1))}} \le C{\mu ^\varepsilon },
			\end{equation*}
			with the constant $C$ independent of $\mu $ and $\varepsilon $. Here ${\widetilde M_{{N_n}}}$ and ${M_{{N_n}}}$ respectively denote the functions that generate the manifolds ${\widetilde{\mathcal M}_{{N_n}}}$ and ${\mathcal M}_{{N_n}}$.
		\end{enumerate}
	\end{definition}
	\par
	Next, we will construct an extension manifold $\widetilde {\mathcal{M}}_{N_n}$, which will be $C_b^1$-close to the IM ${{\mathcal{M}}_{{N_n}}}$. First, we restrict the ${C^{1,\varepsilon }}$-smooth map ${M_{{N_n}}}:{P_{{N_n}}}(H_0^1) \to {Q_{{N_n}}}(H_0^1)$ to the invariant set ${P_{{N_n}}}{{\mathcal M}_{{N_1}}}$. By the Whitney Extension Theorem, if there exists a family of polynomials $\{ J_\xi ^n{M_{{N_n}}}(p),p \in V\} $ that satisfy the compatibility condition \eqref{compatibility condition}, then there exists a ${C^{n,\varepsilon }}$-smooth function ${\widehat M_{{N_n}}}:{P_{{N_n}}}(H_0^1) \to {Q_{{N_n}}}(H_0^1)$ such that ${J_\xi ^n{{\widehat M}_{{N_n}}}( p ) = J_\xi ^n{M_{{N_n}}}( p )}$, for ${p \in {P_{{N_n}}}{{\mathcal M}_{{N_1}}}}$.
	\par
	To ensure the ${C^1}$-norm closeness, for any small parameter $\mu > 0$, we construct a cut-off function ${\rho _\mu } \in {C^\infty }({P_{{N_n}}}(H_0^1),\mathbb{R})$. This function satisfies $\rho_\mu(p) = 0$ when $p$ lies in $O_\mu$ (the $\mu$-neighborhood of $P_{N_n}\mathcal{M}_{N_1}$), and $\rho_\mu(p) = 1$ when $p$ does not lie in $O_{2\mu}$. Moreover, since ${P_{{N_n}}}{M_{{N_1}}}$ is ${C^{1,\varepsilon }}$-smooth, we can require that $| {{\nabla _p}{\rho _\mu }(p)} | \le C{\mu ^{ - 1}}$, where the constant $C$ is independent of $\mu $. Finally, we define the desired function ${\widetilde{M}_{{N_n}}}$ as follows:
	\begin{equation}\label{closeness}
		{\widetilde{M}_{{N_n}}}( p ): = ( {1 - {\rho _\mu }( p )} ){\widehat{M}_{{N_n}}}( p ) + {\rho _\mu }( p )( {{S_{{\mu ^n}}}{M_{{N_n}}}} )( p ),
	\end{equation}
	where ${S_{{\mu ^n}}}$ is a standard mollifying operator:
	\begin{equation*}
		( {{S_{{\mu ^n}}}{M_{{N_n}}}} )( p ): = \int_{{\mathbb{R}^{{N_n}}}} {{\beta _{{\mu ^n}}}( {p - q} ){M_{{N_n}}}( q )dq},
	\end{equation*}
	the kernel ${\beta _{{\mu ^n}}}( p ) = \frac{1}{{{{( {{\mu ^n}} )}^{{N_n}}}}}{\beta _1}( {\frac{p}{{{\mu ^n}}}} )$ and ${\beta _1}( p )$ is a smooth, non-negative function with compact support satisfying $\int_{{\mathbb{R}^{{N_n}}}} {{\beta _1}( p )dp} = 1$ (see \cite{KZ2024} for more details).
	\par
	Notably, for conciseness, we define the polynomials $\{ {Q_{{N_n}}}J_\xi ^nW(p,0),p \in {P_{{N_n}}}{{\mathcal M}_{{N_1}}}\} $ obtained by induction in the subsequent proof as $\{ J_\xi ^n{M_{{N_n}}}(p),p \in {P_{{N_n}}}{{\mathcal M}_{{N_1}}}\} $. We now move to proving Theorem \ref{my theorem}.
	\par
	\begin{proof}[Proof of Theorem \ref{my theorem}] Our primary idea is to apply the inductive method to obtain the Taylor jets of each order using Lemma \ref{my main lemma}, and to verify that the Taylor jets of each order satisfy the compatibility conditions \eqref{compatibility condition} of the Whitney Extension Theorem. Finally, applying equation \eqref{closeness}, we establish the ${C^{n,\varepsilon }}$-smooth extension ${\widetilde{\mathcal M}_{{N_n}}}$ and ${C^{n,\varepsilon }}$-smooth IF.
		\par
		\textbf{Step 1. $\bm{n = 2}$.} According to Proposition \ref{C1 smooth IM}, for any ${p_1},p,\xi \in H_0^1$, the first derivative ${V'_\xi }: = V'({P_{{N_1}}}p,t){P_{{N_1}}}\xi $ solves the problem
		\begin{equation}
			\left\{ {\begin{array}{*{20}{c}}
					{{\partial _t}{V'_\xi } + A{V'_\xi } - ({F'_1}(v){V'_\xi } + {F'_2}(v){V'_\xi }) = 0,}\\
					{\begin{array}{*{20}{c}}
							{{P_{{N_1}}}{V'_\xi }{|_{t = 0}} = {P_{{N_1}}}\xi ,}&{t \le 0},
					\end{array}}
			\end{array}} \right.
		\end{equation}
		and belongs to the space $L_{{\theta _1}}^2({\mathbb{R}_ - },H_0^1)$, where ${\theta _1} = {\lambda _{{N_1}}} + \gamma + \lambda _{{N_1} + 1}^{1/2}{L_1} + {L_2}$. Moreover we have the estimate
		\begin{equation*}
			{\| {V({P_{{N_1}}}{p_1},t) - V({P_{{N_1}}}p,t) - V'({P_{{N_1}}}p,t){P_{{N_1}}}\xi } \|_{L_{(1 + \varepsilon ){\theta _1}}^2({\mathbb{R}_ - },H_0^1)}} \le C\| {{P_{{N_1}}}(p - {p_1})} \|_{H_0^1}^{1 + \varepsilon },
		\end{equation*}
		where $\varepsilon > 0$, $\xi : = p - {p_1}$ and $C$ is independent of ${p_1}$ and ${p}$.
		\par
		Similarly, the first derivative ${W'_\xi }: = W'( {{P_{{N_2}}}p,t} ){P_{{N_2}}}\xi $ of IM ${\mathcal{M}_{{N_2}}}$ solves the problem
		\begin{equation}
			\left\{ {\begin{array}{*{20}{c}}
					{{\partial _t}{W'_\xi } + A{W'_\xi } - ({F'_1}(v){W'_\xi } + {F'_2}(v){W'_\xi }) = 0,}\\
					{\begin{array}{*{20}{c}}
							{{P_{{N_2}}}{W'_\xi }{|_{t = 0}} = {P_{{N_2}}}\xi ,}&{t \le 0},
					\end{array}}
			\end{array}} \right.
		\end{equation}
		and belongs to the space $L_{{\theta _2}}^2({\mathbb{R}_ - },H_0^1)$, where ${\theta _2} = {\lambda _{{N_2}}} + \gamma + \lambda _{{N_2} + 1}^{1/2}{L_1} + {L_2}$. Furthermore for any ${p_1},p \in H_0^1$ we have
		\begin{equation*}
			{\| {W({P_{{N_2}}}{p_1},t) - W({P_{{N_2}}}p,t) - W'({P_{{N_2}}}p,t){P_{{N_1}}}\xi } \|_{L_{(1 + \varepsilon ){\theta _2}}^2({\mathbb{R}_ - },H_0^1)}} \le C\| {{P_{{N_2}}}(p - {p_1})} \|_{H_0^1}^{1 + \varepsilon },
		\end{equation*}
		where $C$ is independent of ${p_1}$ and ${p}$.
		\par
		Since ${\mathcal{M}_{{N_1}}} \subset {\mathcal{M}_{{N_2}}}$, then for every $p \in {\mathcal{M}_{{N_1}}}$, we derive
		\begin{equation}
			v(t) = W(p,t): = W({P_{{N_2}}}p,t) = V(p,t): = V({P_{{N_1}}}p,t).
		\end{equation}
		From this, we obtain,
		\begin{equation}\label{The difference of first-order derivative of V and W}
			{\| {V'({P_{{N_1}}}p,t){P_{{N_1}}}\xi - W'({P_{{N_2}}}p,t){P_{{N_2}}}\xi } \|_{L_{(1 + \varepsilon ){\theta _2}}^2({\mathbb{R}_ - },H_0^1)}} \le C\| {{P_{{N_2}}}(p - {p_1})} \|_{H_0^1}^{1 + \varepsilon }.
		\end{equation}
		\par
		For each $p \in {\mathcal{M}_{{N_1}}}$ and $\xi \in H_0^1$, define the ``second derivative" ${W''_\xi }: = W''(p,t)[\xi ,\xi ]: = W''({P_{{N_2}}}p,t)[{P_{{N_2}}}\xi ,{P_{{N_2}}}\xi ]$ of the trajectory $u( t ) = V( {p,t} ) = W( {p,t} )$ as a solution of the problem:
		\begin{equation}\label{C2 extension IM equation}
			\left\{
			\begin{array}{cr}
				\multicolumn{2}{c}{\partial_t W''_\xi + AW''_\xi - (F'_1(v) W''_\xi + F'_2(v) W''_\xi) = 2 F''_1(v) [V'_\xi, W'_\xi] - F''_1(v) [V'_\xi, V'_\xi]} \\
				& \quad + 2 F''_2(v) [V'_\xi, W'_\xi] - F''_2(v) [V'_\xi, V'_\xi], \\
				\multicolumn{2}{c}{ P_{N_2} W''_\xi |_{t = 0} = 0.}
			\end{array}
			\right.
		\end{equation}
		Since ${V'_\xi } \in L_{{\theta _1}}^2({\mathbb{R}_ - },H_0^1) \cap {C_{{\theta _1}}}({\mathbb{R}_ - },H_0^1)$, ${W'_\xi } \in L_{{\theta _2}}^2({\mathbb{R}_ - },H_0^1) \cap {C_{{\theta _2}}}({\mathbb{R}_ - },H_0^1)$, and given that ${F_1} \in C_b^{n + 1}(H_0^1,H)$ and ${F_2} \in C_b^{n + 1}(H_0^1,H_0^1)$, we obtain
		\begin{equation*}
			\| {{F''_1}(v)[{V'_\xi },{W'_\xi }]} \|_{L_{{\theta _1} + {\theta _2}}^2({\mathbb{R}_ - },H)}^2 \le {C_\theta }\| {{V'_\xi }} \|_{L_{{\theta _1}}^2({\mathbb{R}_ - },H_0^1)}^2\| {{W'_\xi }} \|_{{C_{{\theta _2}}}({\mathbb{R}_ - },H_0^1)}^2 \le {C_\theta }\| {{P_{{N_2}}}\xi } \|_{H_0^1}^4.
		\end{equation*}
		Furthermore, since $2{\theta _1} < {\theta _1} + {\theta _2}$, it follows that $2{F''_1}(v)[{V'_\xi },{W'_\xi }] - {F''_1}(v)[{V'_\xi },{V'_\xi }]$ belongs to $L_{{\theta _1} + {\theta _2}}^2({\mathbb{R}_ - },H)$, while $2{F''_2}(v)[{V'_\xi },{W'_\xi }] - {F''_2}(v)[{V'_\xi },{V'_\xi }]$ belongs to $L_{{\theta _1} + {\theta _2}}^2({\mathbb{R}_ - },H_0^1)$. According to Lemma \ref{my main lemma} (when $i = 2$), there exists a unique solution of problem \eqref{C2 extension IM equation} belonging to the space ${C_{{\theta _1} + {\theta _2}}}({\mathbb{R}_ - },H_0^1) \cap L_{{\theta _1} + {\theta _2}}^2({\mathbb{R}_ - },H_0^1)$, the function ${W''_\xi }$ is well-defined and satisfies
		\begin{equation*}
			\begin{array}{*{20}{c}}
				{{{\| {{W''_\xi }} \|}_{L_{{\theta _1} + {\theta _2}}^2({\mathbb{R}_ - },H_0^1)}} \le {C_\theta }\| {{P_{{N_2}}}\xi } \|_{H_0^1}^2,}&{{{\| {{W''_\xi }} \|}_{{C_{{\theta _1} + {\theta _2}}}({\mathbb{R}_ - },H_0^1)}} \le {C_\theta }\| {{P_{{N_2}}}\xi } \|_{H_0^1}^2.}
			\end{array}
		\end{equation*}
		\par
		Let us establish the desired quadratic polynomial $\xi \to J_\xi ^2W(p,t)$, $p \in {{\mathcal M}_{{N_1}}}$ as follows:
		\begin{equation*}
			\begin{array}{*{20}{c}}
				{J_\xi ^2W(p,t): = V(p,t) + W'(p,t)\xi + \frac{1}{2}W''(p,t)[\xi ,\xi ],}&{\xi \in H_0^1.}
			\end{array}
		\end{equation*}
		\par
		Now, we have obtained the second order ``Taylor jet", and we need to verify that it satisfies the compatibility condition in Whitney Extension Theorem:
		\begin{equation*}
			\begin{array}{*{20}{c}}
				{{{\| {P(\xi + \eta ,p) - P(\xi ,{p_1})} \|}_{H_0^1}} \le C({{({{\| \eta \|}_{H_0^1}} + {{\| \xi \|}_{H_0^1}})}^{2 + \varepsilon }}),}&{\eta : = {p_1} - p,}&{\forall \xi \in H_0^1.}
			\end{array}
		\end{equation*}
		\par
		In fact, the above compatibility condition can take on an equivalent form:
		\begin{equation}\label{n=2 compatibility condition}
			{\| {{P_l}({{\{ \xi \} }^l},p + \eta ) - \sum\nolimits_{k = 0}^{n - l} {\frac{1}{{k!}}{P_{l + k}}([{{\{ \xi \} }^l},{{\{ \eta \} }^k}],p)} } \|_{H_0^1}} \le C\| \xi \|_{H_0^1}^l\| \eta \|_{H_0^1}^{n - l + \varepsilon }
		\end{equation}
		for $l = \{ {0, \cdot \cdot \cdot ,n} \}$ and $n = 2$, (see \cite{KZ2024} for more details). Therefore, we will discuss the following three cases (for simplicity, in the subsequent proofs, we will no longer explicitly distinguish or emphasize which spectral projection each component belongs to):
		\begin{enumerate}[label=\textbf{\roman{enumi}}., leftmargin=0cm]
			\item The zero order compatibility condition ($l=0$, $n=2$):
			\begin{equation*}
				{\| {V({p_1},t) - V(p,t) - W'(p,t)\eta - \frac{1}{2}W''(p,t)[\eta ,\eta ]} \|_{L_{{\theta _1} + {\theta _2} + \varepsilon }^2({\mathbb{R}_ - },H_0^1)}} \le C\| \eta \|_{H_0^1}^{2 + \varepsilon }
			\end{equation*}
			for ${p_1},p \in {{\mathcal M}_{{N_1}}}$, $\eta : = {p_1} - p$ and $\xi \in H_0^1$.
			\par
			Let $R(t): = V({p_1},t) - V(p,t) - W'(p,t)\eta - \frac{1}{2}W''(p,t)[\eta ,\eta ]$. Then, as elementary computations show, $R(t)$ satisfies the problem
			\begin{multline}\label{l=0,n=2 equation}
				{\partial _t}R + AR - ({F'_1}(V(p,t))R + {F'_2}(V(p,t))R) \\
				= {F_1}(V({p_1},t)) - {F_1}(V(p,t)) - {F'_1}(V(p,t))[V({p_1},t) - V(p,t)] \\
				- \frac{1}{2}[2{F''_1}(V(p,t))[V'(p,t)\eta ,W'(p,t)\eta ] - {F''_1}(V(p,t))[V'(p,t)\eta ,V'(p,t)\eta ]] \\
				+ {F_2}(V({p_1},t)) - {F_2}(V(p,t)) - {F'_2}(V(p,t))[V({p_1},t) - V(p,t)] \\
				- \frac{1}{2}[2{F''_2}(V(p,t))[V'(p,t)\eta ,W'(p,t)\eta ] - {F''_2}(V(p,t))[V'(p,t)\eta ,V'(p,t)\eta ]],{ {{P_{{N_2}}}R} |_{t = 0}} = 0,
			\end{multline}
			where we have used that $v(t) = W(p,t) = V(p,t)$, for every $p \in {\mathcal{M}_{{N_1}}}$. Since ${F_1} \in C_b^{2,\varepsilon }(H_0^1,H)$, ${F_2} \in C_b^{2,\varepsilon }(H_0^1,H_0^1)$ and $V \in {C^{1,\varepsilon }}(H_0^1,H_0^1)$, we have
			\begin{equation*}
				V({p_1},t) - V(p,t) = V'(p,t)\eta + O_{(1 + \varepsilon ){\theta _1}}^{H_0^1}(\| \eta \|_{H_0^1}^{1 + \varepsilon }),
			\end{equation*}
			\begin{multline*}
				{F_1}(V({p_1},t)) - {F_1}(V(p,t)) - {F'_1}(V(p,t))[V({p_1},t) - V(p,t)] \\
				= \frac{1}{2}{F''_1}(V(p,t))[V'(p,t)\eta ,V'(p,t)\eta ] + O_{{\theta _1} + {\theta _2} + \varepsilon }^H(\| \eta \|_{H_0^1}^{2 + \varepsilon }),
			\end{multline*}
			and
			\begin{multline*}
				{F_2}(V({p_1},t)) - {F_2}(V(p,t)) - {F'_2}(V(p,t))[V({p_1},t) - V(p,t)] \\
				= \frac{1}{2}{F''_2}(V(p,t))[V'(p,t)\eta ,V'(p,t)\eta ] + O_{{\theta _1} + {\theta _2} + \varepsilon }^{H_0^1}(\| \eta \|_{H_0^1}^{2 + \varepsilon })
			\end{multline*}
			for some small $\varepsilon$, where $\eta : = {p_1} - p$. Then the right-hand side of equation \eqref{l=0,n=2 equation} equals to
			\begin{multline*}
				{F''_1}(V(p,t))[V'(p,t)\eta ,V'(p,t)\eta - W'(p,t)\eta ] + O_{{\theta _1} + {\theta _2} + \varepsilon }^H(\| \eta \|_{H_0^1}^{2 + \varepsilon }) \\
				+ {F''_2}(V(p,t))[V'(p,t)\eta ,V'(p,t)\eta - W'(p,t)\eta ] + O_{{\theta _1} + {\theta _2} + \varepsilon }^{H_0^1}(\| \eta \|_{H_0^1}^{2 + \varepsilon }).
			\end{multline*}
			Using \eqref{The difference of first-order derivative of V and W}, we obtain:
			\begin{equation*}
				\begin{array}{*{20}{c}}
					{{F''_1}(V(p,t))[V'(p,t)\eta ,V'(p,t)\eta - W'(p,t)\eta ] = O_{{\theta _1} + {\theta _2} + \varepsilon }^H(\| \eta \|_{H_0^1}^{2 + \varepsilon }),}\\
					{{F''_2}(V(p,t))[V'(p,t)\eta ,V'(p,t)\eta - W'(p,t)\eta ] = O_{{\theta _1} + {\theta _2} + \varepsilon }^{H_0^1}(\| \eta \|_{H_0^1}^{2 + \varepsilon }).}
				\end{array}
			\end{equation*}
			Therefore $R(t)$ satisfies the problem
			\begin{equation*}
				\left\{ {\begin{array}{*{20}{c}}
						{{\partial _t}R + AR - ({F'_1}(V(p,t))R + {F'_2}(V(p,t))R) = O_{{\theta _1} + {\theta _2} + \varepsilon }^H(\| \eta \|_{H_0^1}^{2 + \varepsilon }) + O_{{\theta _1} + {\theta _2} + \varepsilon }^{H_0^1}(\| \eta \|_{H_0^1}^{2 + \varepsilon }),}\\
						{{P_{{N_2}}}R{|_{t = 0}} = 0.}
				\end{array}} \right.
			\end{equation*}
			Thus, the Lemma \ref{my main lemma} gives us that
			\begin{equation}
				{\| R \|_{L_{{\theta _1} + {\theta _2} + \varepsilon }^2({\mathbb{R}_ - },H_0^1)}} \le {C_\theta }\| \eta \|_{H_0^1}^{2 + \varepsilon }
			\end{equation}
			for some small $\varepsilon$.
			\item The first order compatibility condition ($l=1$, $n=2$):
			\begin{equation*}
				{\| {W'({p_1},t)[\xi ] - W'(p,t)[\xi ] - W''(p,t)[\eta ,\xi ]} \|_{L_{{\theta _1} + {\theta _2} + \varepsilon }^2({\mathbb{R}_ - },H_0^1)}} \le C\| \eta \|_{H_0^1}^{1 + \varepsilon }\| \xi \|_{H_0^1}^1
			\end{equation*}
			for ${p_1},p \in {{\mathcal M}_{{N_1}}}$ and $\eta : = {p_1} - p$.
			\par
			To this end, we expand the difference $\bar w(t): = W'({p_1},t)\xi - W'(p,t)\xi $, ${p_1},p \in {\mathcal{M}_{{N_1}}}$ in terms of $\eta : = {p_1} - p$. By the definition of $W'$, this function satisfies the problem:
			\begin{multline*}
				{\partial _t}\bar w + A\bar w - ({F'_1}(V({p_1},t))\bar w + {F'_2}(V({p_1},t))\bar w) = ({F'_1}(V({p_1},t)) - {F'_1}(V(p,t)))W'(p,t)\xi \\
				+ ({F'_2}(V({p_1},t)) - {F'_2}(V(p,t)))W'(p,t)\xi ,{ {{P_{{N_2}}}\bar w} |_{t = 0}} = 0.
			\end{multline*}
			Since ${F_1} \in C_b^{2,\varepsilon }(H_0^1,H)$, ${F_2} \in C_b^{2,\varepsilon }(H_0^1,H_0^1)$ and $V \in {C^{1,\varepsilon }}(H_0^1,H_0^1)$, we have 
			\begin{equation*}
				\begin{array}{*{20}{c}}
					{{F'_1}(V({p_1},t)) - {F'_1}(V(p,t)))W'(p,t)\xi = {F''_1}[V'({p_1},t)\eta ,W'(p,t)\xi ] + O_{{\theta _1} + {\theta _2} + \varepsilon }^H(\| \eta \|_{H_0^1}^{1 + \varepsilon }\| \xi \|_{H_0^1}^1),}\\
					{{F'_2}(V({p_1},t)) - {F'_2}(V(p,t)))W'(p,t)\xi = {F''_2}[V'({p_1},t)\eta ,W'(p,t)\xi ] + O_{{\theta _1} + {\theta _2} + \varepsilon }^{H_0^1}(\| \eta \|_{H_0^1}^{1 + \varepsilon }\| \xi \|_{H_0^1}^1),}
				\end{array}
			\end{equation*}
			where ${p_1},p \in {{\mathcal M}_{{N_1}}}$, $\eta : = {p_1} - p$ and $\xi \in H_0^1$.
			\par
			Next, we analyze the bilinear forms ${F''_1}[V'({p_1},t)\eta ,W'(p,t)\xi ]$ and ${F''_2}[V'({p_1},t)\eta ,W'(p,t)\xi ]$ (w.r.t. $\eta ,\xi$). Note that, in contrast to the case where the IM is $C^2$, this form is even not symmetric, so it should be corrected. Specifically, we begin by establishing two identities:
			\begin{multline*}
				{F''_1}[V'({p_1},t)\eta ,W'(p,t)\xi ] = \{ {F''_1}(V(p,t))[V'(p,t)\eta ,W'(p,t)\xi ] + {F''_1}(V(p,t))[V'(p,t)\xi ,W'(p,t)\eta ] \\
				- {F''_1}(V(p,t))[V'(p,t)\eta ,V'(p,t)\xi ]\} - \{ {F''_1}(V(p,t))[V'(p,t)\xi ,W'(p,t)\eta - V'(p,t)\eta ]\},
			\end{multline*}
			and
			\begin{multline*}
				{F''_2}[V'({p_1},t)\eta ,W'(p,t)\xi ] = \{ {F''_2}(V(p,t))[V'(p,t)\eta ,W'(p,t)\xi ] + {F''_2}(V(p,t))[V'(p,t)\xi ,W'(p,t)\eta ] \\
				- {F''_2}(V(p,t))[V'(p,t)\eta ,V'(p,t)\xi ]\} - \{ {F''_2}(V(p,t))[V'(p,t)\xi ,W'(p,t)\eta - V'(p,t)\eta ]\},
			\end{multline*}
			where the first term appearing in the right-hand side of these equations constitutes the symmetric bilinear form associated with the quadratic forms:
			\begin{equation*}
				2{F''_1}(V(p,t))[V'(p,t)\xi ,W'(p,t)\xi ] - {F''_1}(V(p,t))[V'(p,t)\xi ,V'(p,t)\xi ],
			\end{equation*}
			and
			\begin{equation*}
				2{F''_2}(V(p,t))[V'(p,t)\xi ,W'(p,t)\xi ] - {F''_2}(V(p,t))[V'(p,t)\xi ,V'(p,t)\xi ],
			\end{equation*}
			which defines $W''$ through problem \eqref{C2 extension IM equation}. At the same time, based on estimate \eqref{The difference of first-order derivative of V and W}, we obtain the following estimates for the second terms:
			\begin{equation*}
				\begin{array}{*{20}{c}}
					{{F''_1}(V(p,t))[V'(p,t)\xi ,W'(p,t)\eta - V'(p,t)\eta ] = O_{{\theta _1} + {\theta _2} + \varepsilon }^H(\| \eta \|_{H_0^1}^{1 + \varepsilon }\| \xi \|_{H_0^1}^1),}\\
					{{F''_2}(V(p,t))[V'(p,t)\xi ,W'(p,t)\eta - V'(p,t)\eta ] = O_{{\theta _1} + {\theta _2} + \varepsilon }^{H_0^1}(\| \eta \|_{H_0^1}^{1 + \varepsilon }\| \xi \|_{H_0^1}^1).}
				\end{array}
			\end{equation*}
			\par
			Thus, by Lemma \ref{my main lemma}, we have:
			\begin{equation}
				{\| {\bar w(t) - W''(p,t)[\eta ,\xi ]} \|_{L_{{\theta _1} + {\theta _2} + \varepsilon }^2({\mathbb{R}_ - },H_0^1)}} \le C\| \eta \|_{H_0^1}^{1 + \varepsilon }\| \xi \|_{H_0^1}^1
			\end{equation}
			for ${p_1},p \in {{\mathcal M}_{{N_1}}}$ and $\xi \in H_0^1$.
			\item The second order compatibility condition ($l=2$, $n=2$):
			\begin{equation*}
				{\| {W''({p_1},t)[\xi ,\xi ] - W''(p,t)[\xi ,\xi ]} \|_{L_{{\theta _1} + {\theta _2} + \varepsilon }^2({\mathbb{R}_ - },H_0^1)}} \le C\| \xi \|_{H_0^1}^2\| \eta \|_{H_0^1}^\varepsilon
			\end{equation*}
			for ${p_1},p \in {{\mathcal M}_{{N_1}}}$, $\eta : = {p_1} - p$ and $\xi \in H_0^1$.
			\par
			Let $R(t): = W''({p_1},t)[\xi ,\xi ] - W''(p,t)[\xi ,\xi ]$. Then, $R(t)$ satisfies the equation:
			\begin{equation*}
				{\partial _t}R + AR - ({F'_1}(V({p_1},t))R + {F'_2}(V({p_1},t))R) = {h_{1,1}} + {h_{1,2}} + {h_{1,3}} + {h_{2,1}} + {h_{2,2}} + {h_{2,3}},
			\end{equation*}
			where ${h_{1,3}}: = [{F'_1}(V({p_1},t)) - {F'_1}(V(p,t))][{W''_\xi }]$, ${h_{2,3}}: = [{F'_2}(V({p_1},t)) - {F'_2}(V(p,t))][{W''_\xi }]$,
			\begin{multline*}
				{h_{1,1}}: = 2{F''_1}(V({p_1},t))[{V'_{1,\xi }},{W'_{1,\xi }} - {W'_\xi }] + 2{F''_1}(V({p_1},t))[{V'_{1,\xi }} - {V'_\xi },{W'_\xi }] \\
				+ 2[{F''_1}(V({p_1},t)) - 2{F''_1}(V(p,t))][{V'_\xi },{W'_\xi }],
			\end{multline*}
			\begin{multline*}
				{h_{1,2}}: = - \{ {F''_1}(V({p_1},t))[{V'_{1,\xi }},{V'_{1,\xi }} - {V'_\xi }] + {F''_1}(V({p_1},t))[{V'_{1,\xi }} - V',{V'_\xi }] \\
				+ [{F''_1}(V({p_1},t)) - {F''_1}(V(p,t))][{V'_\xi },{V'_\xi }]\},
			\end{multline*}
			\begin{multline*}
				{h_{2,1}}: = 2{F''_2}(V({p_1},t))[{V'_{1,\xi }},{W'_{1,\xi }} - {W'_\xi }] + 2{F''_2}(V({p_1},t))[{V'_{1,\xi }} - {V'_\xi },{W'_\xi }] \\
				+ 2[{F''_2}(V({p_1},t)) - 2{F''_2}(V(p,t))][{V'_\xi },{W'_\xi }],
			\end{multline*}
			and
			\begin{multline*}
				{h_{2,2}}: = - \{ {F''_2}(V({p_1},t))[{V'_{1,\xi }},{V'_{1,\xi }}] - \{ {F''_2}(V({p_1},t))[{V'_{1,\xi }},{V'_{1,\xi }} - {V'_\xi }] + {F''_2}(V({p_1},t))[{V'_{1,\xi }} - V',{V'_\xi }] \\
				+ [{F''_2}(V({p_1},t)) - {F''_2}(V(p,t))][{V'_\xi },{V'_\xi }]\}.
			\end{multline*}
			Here ${V'_\xi }: = V'(p,t)\xi$, ${V''_\xi }: = V''(p,t)[\xi ,\xi ]$, ${V'_{1,\xi }}: = V'({p_1},t)\xi $, ${W'_\xi }: = W'(p,t)\xi $, ${W''_\xi }: = W''(p,t)[\xi ,\xi ]$, ${W'_{1,\xi }}: = W'({p_1},t)\xi $. Using ${F_1} \in C_b^{2,\varepsilon }(H_0^1,H)$, ${F_2} \in C_b^{2,\varepsilon }(H_0^1,H_0^1)$ and $V \in {C^{1,\varepsilon }}(H_0^1,H_0^1)$ in conjunction with Lemma \ref{my main lemma}, $R(t)$ satisfies the following estimate:
			\begin{equation*}
				{\| R \|_{L_{{\theta _1} + {\theta _2} + \varepsilon }^2({\mathbb{R}_ - },H_0^1)}} \le {C_\theta }\| \xi \|_{H_0^1}^2\| \eta \|_{H_0^1}^\varepsilon,
			\end{equation*}
			where ${C_\theta }$ is independent of $\xi $ and $\eta $.
		\end{enumerate}
		\par
		This finishes the verification of the compatibility conditions. With the help of equation \eqref{closeness}, we constructed the required manifold ${{\widetilde{\mathcal{M}}_{{N_2}}}}$. This concludes the proof of the theorem in the case of $n = 2$.
		\par
		\textbf{Step 2. $\bm{n > 2}$.} We will proceed by induction with respect to $n$.
		\par
		Suppose that for any fixed $( {n - 1} ) \in \mathbb{N}^+$, $p \in {P_{{N_{n - 1}}}}{\mathcal{M}_{{N_1}}}$, we have constructed the $( {n - 1} )$th ``Taylor jet" $J_\xi ^{n - 1}V( {p,t} )$, which satisfies the following compatibility condition:
		\begin{equation}\label{(n-1)th compatibility condition}
			J_\xi ^{n - 1}V({p_1},t) - J_{\xi + \eta }^{n - 1}V(p,t) = O_{{\theta _{n - 1}} + (n - 2){\theta _{n - 2}} + \varepsilon }^{H_0^1}({({\| \eta \|_{H_0^1}} + {\| \xi \|_{H_0^1}})^{n - 1 + \varepsilon }}).
		\end{equation}
		Here, $\xi \in {P_{{N_{n - 1}}}}H_0^1$, $\eta : = {p_1} - p$, $\varepsilon > 0$, the exponents ${\theta _i} = {\lambda _{{N_i}}} + \gamma + \lambda _{{N_i} + 1}^{1/2}{L_1} + {L_2}$ satisfy ${\theta _i} + (j - 1){\theta _{i - 1}} < {\lambda _{{N_{i + 1}}}} - (\lambda _{{N_i} + 1}^{1/2}{L_1} + {L_2})$, for $i = 1,2, \cdot \cdot \cdot ,n - 1$. Rewriting equation \eqref{(n-1)th compatibility condition} in the form of truncated jets gives:
		\begin{equation}\label{(n-1)th compatibility condition truncated jets}
			j_\xi ^{n - 1}V({p_1},t) + j_\eta ^{n - 1}V(p,t) - j_{\xi + \eta }^{n - 1}V(p,t) = O_{(n - 1){\theta _{n - 1}} + \varepsilon }^{H_0^1}({({\| \eta \|_{H_0^1}} + {\| \xi \|_{H_0^1}})^{n - 1 + \varepsilon }}),
		\end{equation}
		where we have used ${\theta _{n - 2}} < {\theta _{n - 1}}$. Additionally, we induction assumption that \eqref{(n-1)th compatibility condition truncated jets} holds for all $m \le n - 1$, i.e.,
		\begin{equation}
			J_\xi ^mV({p_1},t) - J_{\xi + \eta }^mV(p,t) = O_{m{\theta _{n - 1}} + \varepsilon }^{H_0^1}({({\| \eta \|_{H_0^1}} + {\| \xi \|_{H_0^1}})^{m + \varepsilon }}).
		\end{equation}
		\par
		Next, we define the $n$th ``Taylor jet" $J_\xi ^{n}W( {p,t} )$ of function $W( {p,t} )$ using the $n$th spectral gap as follows:
		\begin{equation}\label{extension Taylor jets}
			J_\xi ^nW( {p,t} ) = W( {p,t} ) + \sum_{d = 1}^n {\frac{1}{{d!}}} {W^{( d )}}( {p,t} )[{{{\{ \xi \}}^d}}],
		\end{equation}
		where $\xi \in {P_{{N_n}}}(H_0^1)$ and $p \in {P_{{N_n}}}{\mathcal{M}_{{N_1}}}$. Recalling the case $n = 2$, we construct the required jet \eqref{extension Taylor jets} as a backward solution of the problem:
		\begin{equation}
			\left\{ {\begin{array}{*{20}{c}}
					{{\partial _t}J_\xi ^nW(p,t) + AJ_\xi ^nW(p,t) = F_1^{[n]}(p,\xi ,t) + F_2^{[n]}(p,\xi ,t),}\\
					{{P_{{N_n}}}J_\xi ^nW(p,t){|_{t = 0}} = {P_{{N_n}}}(p + \xi ),}
			\end{array}} \right.
		\end{equation}
		where
		\begin{multline}\label{Expansion of F 2}
			F_i^{[n]}(p,\xi ,t): = {F_i}(W(p,t)) + F'(W(p,t))j_\xi ^nW(p,t) \\
			+ \sum\limits_{d = 2}^n {\frac{1}{{d!}}} (dF_i^{(d)}(W(p,t){[\{ j_\xi ^{n - 1}V(p,t))\} ^{d - 1}},j_\xi ^{n - 1}W(p,t)] \\
			- (d - 1)F_i^{(d)}(W(p,t))[{\{ j_\xi ^{n - 1}V(p,t)\} ^d}]), i = 1,2.
		\end{multline}
		Notably, the symbol ``$[n]$" indicates that all terms of order greater than $n$ are removed from the right-hand side. Thus, $F_i^{[n]}$ denotes a polynomial of degree at most $n$ in $\xi$ (with $\xi \in {P_{{N_n}}}(H_0^1)$). This process of dropping terms means that we replace
		\begin{multline}\label{Replacement Technique}
			[{\{ j_\xi ^{n - 1}V(p,t)\} ^{d - 1}},j_\xi ^{n - 1}W(p,t)] \to \\
			\sum\limits_{{n_1} + \cdot \cdot \cdot + {n_d} \le n} {{B_{{n_1}, \cdot \cdot \cdot ,{n_d}}}} \{ j_\xi ^{{n_1}}V(p,t), \cdot \cdot \cdot ,j_\xi ^{{n_{d - 1}}}V(p,t),j_\xi ^{{n_d}}W(p,t)\} ,
		\end{multline}
		where ${n_i} \in {\mathbb{N}^ + }$ and the coefficients ${B_{{n_1}, \cdot \cdot \cdot ,{n_d}}} \in \mathbb{R}$ are chosen such that polynomials in the left and right-hand side of \eqref{Replacement Technique} coincide up to order ${\{ \xi \} ^n}$ inclusively. By applying the formulas for the higher order chain rule (Fa\`{a} di Bruno type formulas, see e.g., \cite{HJ2014}), we can derive explicit expressions for these coefficients. Since these expressions are lengthy and not crucial for the subsequent content, so we omit them (see \cite{KZ2024, NLS2026} for more details). Moreover, the term $[{\{ j_\xi ^{n - 1}V(p,t)\} ^d}]$ is treated in a similar way.
		\par
		Expanding \eqref{Expansion of F 2} in series with respect to $\xi $, we obtain recursive problems for finding the derivatives $W_\xi ^{( i )}: = {W^{( i )}}( {p,t} )[{\{ \xi \} ^i}]$, $i = 2, \cdot \cdot \cdot ,n$:
		\begin{equation*}
			\left\{
			\begin{array}{cr}
				\multicolumn{2}{c}{{\partial _t}W_\xi ^{(i)} + AW_\xi ^{(i)} - ({F'_1}(W(p,t))W_\xi ^{(i)} + {F'_2}(W(p,t))W_\xi ^{(i)}) = {\phi _1}(j_\xi ^{i - 1}V,j_\xi ^{i - 1}W)} \\
				& + {\phi _2}(j_\xi ^{i - 1}V,j_\xi ^{i - 1}W), \\
				\multicolumn{2}{c}{{{ {{P_{{N_n}}}W_\xi ^{(i)}} |}_{t = 0}} = 0,}
			\end{array}
			\right.
		\end{equation*}
		where ${{\phi _1}}$ and ${{\phi _2}}$ are smooth maps from $H_0^1$ to $H$ and $H_0^1$ to $H_0^1$. These maps are polynomials of order $(i - 1)$ in $\xi $ and do not contain derivative terms $W_\xi ^{(l)}$ of order $l \ge i$. Thus, the derivatives $W_\xi ^{(l)}$ of each order can be determined recursively. For $i = n$, according to the selection of ${\theta _{n - 1}}$ and ${\theta _n}$, ${\phi _1}(j_\xi ^{i - 1}V,j_\xi ^{i - 1}W)$ belongs to the weighted space $L_{{\theta _n} + ( {i - 1} ){\theta _{n - 1}}}^2( {{\mathbb{R}_ - },H} )$, and ${\phi _2}(j_\xi ^{i - 1}V,j_\xi ^{i - 1}W)$ belongs to the weighted space $L_{{\theta _n} + ( {i - 1} ){\theta _{n - 1}}}^2( {{\mathbb{R}_ - },H_0^1} )$. Through Lemma \ref{my main lemma}, we consequently establish the existence and uniqueness of homogeneous polynomials $W_\xi ^{( i )}$ that satisfying
		\begin{equation}
			\begin{array}{*{20}{c}}
				{{{\| {W_\xi ^{(i)}} \|}_{L_{{\theta _n} + (i - 1){\theta _{n - 1}}}^2({\mathbb{R}_ - },H_0^1)}} \le {C_{L,\theta }}\| \xi \|_{H_0^1}^i,}&{i = 1, \cdot \cdot \cdot ,n.}
			\end{array}
		\end{equation}
		\par
		To complete the proof of the theorem, it suffices to verify that the ``Taylor jet" $J_\xi ^{n}W ({p,t})$ \eqref{extension Taylor jets} satisfies the $n$th compatibility condition. We proceed by induction with respect to the order $m \le n$.
		\par
		For $m = 1$, due to the ${C^{1,\varepsilon }}$-smoothness of the functions $W( {p,t} )$, the first order compatibility conditions are straightforwardly satisfied.
		\par
		For $m = n - 1$, assume that the ${m_1}$th order compatibility conditions hold for all ${m_1} \le m$
		\begin{equation}\label{mth order conditions}
			J_\xi ^{{m_1}}W({p_1},t) - J_{\eta + \xi }^{{m_1}}W(p,t) = O_{{\theta _n} + ({m_1} - 1){\theta _{n - 1}} + \varepsilon }^{H_0^1}({({\| \eta \|_{H_0^1}} + {\| \xi \|_{H_0^1}})^{{m_1} + \varepsilon }}),
		\end{equation}
		where $\xi \in {P_{{N_n}}}(H_0^1)$, ${p_1},p \in {P_{{N_n}}}{\mathcal{M}_{{N_1}}}$, $\varepsilon > 0$ and $\eta : = {p_1} - p$.
		\par
		Subsequently, we will establish the $( {m + 1} )$th order compatibility condition in the following Steps.
		\par
		\textbf{Step 3.} We aim to estimate $J_\xi ^{m + 1}W( {{p_1},t} ) - J_{\eta + \xi }^{m + 1}W( {p,t} )$. To start with, we investigate the compatibility conditions under the specific case of $\xi = 0$. By leveraging $W( {p,t} ) = V( {p,t} ): = V( {{P_{{N_{n - 1}}}}p,t} )$ for all $p \in {P_{{N_n}}}{\mathcal{M}_{{N_1}}}$ together with the analogue of \eqref{mth order conditions} for the already constructed jets $J_\xi ^mV( {p,t} )$, we derive the following:
		\begin{align*}
			V({p_1},t) = W({p_1},t) &= V(p,t) + j_\eta ^{{m_1}}V(p,t) + O_{{m_1}{\theta _{n - 1}} + \varepsilon }^{H_0^1}(\| \eta \|_{H_0^1}^{{m_1} + \varepsilon }) \\
			&= W(p,t) + j_\eta ^{{m_1}}W(p,t) + O_{{\theta _n} + ({m_1} - 1){\theta _{n - 1}} + \varepsilon }^{H_0^1}(\| \eta \|_{H_0^1}^{{m_1} + \varepsilon }).
		\end{align*}
		Therefore, we have
		\begin{align}\label{replace v}
			v: & = V({p_1},t) - V(p,t) = j_\eta ^{{m_1}}V(p,t) + O_{{m_1}{\theta _{n - 1}} + \varepsilon }^{H_0^1}(\| \eta \|_{H_0^1}^{{m_1} + \varepsilon }) \\ \nonumber
			& = W({p_1},t) - W(p,t) = j_\eta ^{{m_1}}W(p,t) + O_{{\theta _n} + ({m_1} - 1){\theta _{n - 1}} + \varepsilon }^{H_0^1}(\| \eta \|_{H_0^1}^{{m_1} + \varepsilon }).
		\end{align}
		\par
		Subsequently, we need to derive the special $(m + 1)$th order compatibility condition when $\xi = 0$. Let us define $R(t): = v - j_\eta ^{m + 1}W(p,t)$. Then, this function satisfies the following problem:
		\begin{equation*}
			\left\{ {\begin{array}{*{20}{c}}
					{{\partial _t}R + AR = {F_1}(W({p_1},t)) - F_1^{[m + 1]}(p,\eta ,t) + {F_2}(W({p_1},t)) - F_2^{[m + 1]}(p,\eta ,t),}\\
					{{P_{{N_n}}}R{|_{t = 0}} = 0.}
			\end{array}} \right.
		\end{equation*}
		\par
		Next, we analyze the expressions $F_1^{[m + 1]}(p,\eta ,t)$ and $F_2^{[m + 1]}(p,\eta ,t)$. By combining the equation \eqref{replace v} with replacement techniques \eqref{Replacement Technique}, we can substitute $j_\eta ^mV(p,t)$ and $j_\eta ^mW(p,t)$ for $v$ in all terms in $F_1^{[m + 1]}(p,\eta ,t)$ and $F_2^{[m + 1]}(p,\eta ,t)$ (the error will be of order $\| \eta \|_{H_0^1}^{m + 1 + \varepsilon }$). Then we obtain
		\begin{multline*}
			F_1^{[m + 1]}(p,\eta ,t) = {F_1}(W(p,t)) + {F'_1}(W(p,t))[j_\eta ^{m + 1}W(p,t)] \\
			+ \sum\limits_{d = 2}^{m + 1} {\frac{1}{{d!}}(F_1^{(d)}(W(p,t))[{{\{ v\} }^d}])} + O_{{\theta _n} + m{\theta _{n - 1}} + \varepsilon }^H(\| \eta \|_{H_0^1}^{m + 1 + \varepsilon }),
		\end{multline*}
		and
		\begin{multline*}
			F_2^{[m + 1]}(p,\eta ,t) = {F_2}(W(p,t)) + {F'_2}(W(p,t))[j_\eta ^{m + 1}W(p,t)] \\
			+ \sum\limits_{d = 2}^{m + 1} {\frac{1}{{d!}}(F_2^{(d)}(W(p,t))[{{\{ v\} }^d}])} + O_{{\theta _n} + m{\theta _{n - 1}} + \varepsilon }^{H_0^1}(\| \eta \|_{H_0^1}^{m + 1 + \varepsilon }).
		\end{multline*}
		Notably, we cannot substitute $j_\eta ^{m + 1}W( {p,t} )$ for $v$, as it requires to satisfy the compatibility condition of order $m + 1$. Indeed, let us consider the terms $[{\{ j_\eta ^mV(p,t)\} ^{d - 1}},j_\eta ^mW(p,t)]$ and $[{\{ j_\eta ^mV(p,t)\} ^d}]$ in $F_1^{[m + 1]}(p,\eta ,t)$ and $F_2^{[m + 1]}(p,\eta ,t)$. Using \eqref{Replacement Technique}:
		\par
		\begin{equation*}
			[{\{ j_\eta ^mV(p,t)\} ^{d - 1}},j_\eta ^mW(p,t)] \to \sum\limits_{\mathop {{n_1} + \cdot \cdot \cdot + {n_d} \le m + 1}\limits_{{n_i} \in {\mathbb{N}^ + }} } {B_{{n_1}, \cdot \cdot \cdot ,{n_d}}^1\{ j_\eta ^{{n_1}}V(p,t), \cdot \cdot \cdot ,j_\eta ^{{n_{d - 1}}}V(p,t),j_\eta ^{{n_d}}W(p,t)\} ,}
		\end{equation*}
		and employing the assumptions \eqref{replace v}, it can be shown that the growth exponent of the remainder does not exceed:
		\begin{equation*}
			( {{n_1} + \cdot \cdot \cdot + {n_{d - 1}}} ){\theta _{n - 1}} + {\theta _n} + ( {{n_d} - 1} ){\theta _{n - 1}} + \varepsilon \le {\theta _n} + m{\theta _{n - 1}} + \varepsilon .
		\end{equation*}
		Here we used conditions \eqref{replace v}. Similarly, using the following substitution technique and condition \eqref{replace v} for $[{\{ j_\eta ^mV(p,t)\} ^d}]$, 
		\begin{equation*}
			[{\{ j_\xi ^{n - 1}V(p,t)\} ^d}] \to \sum\limits_{\mathop {{n_1} + \cdot \cdot \cdot + {n_d} \le m + 1}\limits_{{n_i} \in {\mathbb{N}^ + }} } {B_{{n_1}, \cdot \cdot \cdot ,{n_d}}^2\{ j_\eta ^{{n_1}}V(p,t), \cdot \cdot \cdot ,j_\eta ^{{n_d}}V(p,t)\} } ,
		\end{equation*}
		we conclude that the growth exponent of the remainder is no larger than:
		\begin{equation*}
			( {{n_1} + \cdot \cdot \cdot + {n_d}} ){\theta _{n - 1}} + \varepsilon \le {\theta _n} + m{\theta _{n - 1}} + \varepsilon ,
		\end{equation*}
		where we have utilized the fact that ${\theta _{n - 1}} < {\theta _n}$.
		\par
		By applying the Taylor Theorem to ${F_1} \in C_b^{m + 1,\varepsilon }$ and ${F_2} \in C_b^{m + 1,\varepsilon }$ together with the estimate \eqref{Estimation of V} for $v$, we deduce that:
		\begin{equation*}
			\begin{array}{*{20}{c}}
				{{F_1}(W({p_1},t)) - F_1^{[m + 1]}(p,\eta ,t) = {F'_1}(W(p,t))R + O_{{\theta _n} + m{\theta _{n - 1}} + \varepsilon }^H(\| \eta \|_{H_0^1}^{m + 1 + \varepsilon }),}\\
				{{F_2}(W({p_1},t)) - F_2^{[m + 1]}(p,\eta ,t) = {F'_2}(W(p,t))R + O_{{\theta _n} + m{\theta _{n - 1}} + \varepsilon }^{H_0^1}(\| \eta \|_{H_0^1}^{m + 1 + \varepsilon }).}
			\end{array}
		\end{equation*}
		Therefore, the function $R$ satisfies the equation:
		\begin{equation*}
			\left\{
			\begin{array}{cr}
				\multicolumn{2}{c}{{\partial _t}R + AR - ({F'_1}(W(p,t))R + {F'_2}(W(p,t))R) = O_{{\theta _n} + m{\theta _{n - 1}} + \varepsilon }^H(\| \eta \|_{H_0^1}^{m + 1 + \varepsilon })} \\
				& \quad + O_{{\theta _n} + m{\theta _{n - 1}} + \varepsilon }^{H_0^1}(\| \eta \|_{H_0^1}^{m + 1 + \varepsilon }), \\
				\multicolumn{2}{c}{P_{{N_n}}R|_{t = 0} = 0.}
			\end{array}
			\right.
		\end{equation*}
		\par
		Additionally, based on the selection of ${\theta _{n - 1}}$ and ${\theta _n}$, we can use Lemma \ref{my main lemma} to derive the following estimate:
		\begin{equation*}
			{\| R \|_{L_{{\theta _n} + m{\theta _{n - 1}} + \varepsilon }^2({\mathbb{R}_ - },H_0^1)}} \le C\| \eta \|_{H_0^1}^{m + 1 + \varepsilon }.
		\end{equation*}
		As a result, we have
		\begin{equation}\label{Higher order relationship of v=w}
			v - j_\eta ^{m + 1}W(p,t) = O_{{\theta _n} + m{\theta _{n - 1}} + \varepsilon }^{H_0^1}(\| \eta \|_{H_0^1}^{m + 1 + \varepsilon })
		\end{equation}
		for all ${p_1},p \in {P_{{N_n}}}{\mathcal{M}_{{N_1}}}$ and $\eta : = {p_1} - p$. Now we replace $j_\eta ^{m + 1}W( {p,t} )$ with $v$.
		\par
		With the help of equation \eqref{Higher order relationship of v=w}, we obtain the following result:
		\begin{multline*}
			J_\xi ^{m + 1}W({p_1},t) - J_{\xi + \eta }^{m + 1}W(p,t) = j_\xi ^{m + 1}W({p_1},t) + j_\eta ^{m + 1}W(p,t) - j_{\xi + \eta }^{m + 1}W(p,t) \\
			+ O_{{\theta _n} + m{\theta _{n - 1}} + \varepsilon }^{H_0^1}({({\| \eta \|_{H_0^1}} + {\| \xi \|_{H_0^1}})^{m + 1 + \varepsilon }}).
		\end{multline*}
		\textbf{Step 4.} We shall estimate $F_i^{[m + 1]}({p_1},\xi ,t) - F_i^{[m + 1]}( {p,\xi + \eta ,t} )$, for $i=1,2$. By leveraging Lemma A.2 in \cite{KZ2024}, we derive the following estimations:
		\begin{multline*}
			F_1^{[m + 1]}({p_1},\xi ,t) - F_1^{[m + 1]}( {p,\xi + \eta ,t} ) = {F'_1}(V(p,t))[j_\xi ^{m + 1}W({p_1},t) + j_\eta ^{m + 1}W(p,t) - j_{\xi + \eta }^{m + 1}W(p,t)] \\
			+ O_{{\theta _n} + m{\theta _{n - 1}} + \varepsilon }^H({({\| \eta \|_{H_0^1}} + {\| \xi \|_{H_0^1}})^{m + 1 + \varepsilon }}),
		\end{multline*}
		and
		\begin{multline*}
			F_2^{[m + 1]}({p_1},\xi ,t) - F_2^{[m + 1]}( {p,\xi + \eta ,t} ) = {F'_2}(V(p,t))[j_\xi ^{m + 1}W({p_1},t) + j_\eta ^{m + 1}W(p,t) - j_{\xi + \eta }^{m + 1}W(p,t)] \\
			+ O_{{\theta _n} + m{\theta _{n - 1}} + \varepsilon }^{H_0^1}({({\| \eta \|_{H_0^1}} + {\| \xi \|_{H_0^1}})^{m + 1 + \varepsilon }})
		\end{multline*}
		for $\xi \in {P_{{N_n}}}(H_0^1)$, ${p_1},p \in {P_{{N_n}}}{\mathcal{M}_{{N_1}}}$, $\varepsilon > 0$ and $\eta : = {p_1} - p$.
		\par
		Accordingly, using the conclusion from Step 3, we derive:
		\begin{multline*}
			F_1^{[m + 1]}({p_1},\xi ,t) - F_1^{[m + 1]}( {p,\xi + \eta ,t} ) \\
			= {F'_1}(V(p,t))[J_\xi ^{m + 1}W({p_1},t) - J_{\xi + \eta }^{m + 1}W(p,t)] + O_{{\theta _n} + m{\theta _{n - 1}} + \varepsilon }^H({({\| \eta \|_{H_0^1}} + {\| \xi \|_{H_0^1}})^{m + 1 + \varepsilon }}),
		\end{multline*}
		and
		\begin{multline*}
			F_2^{[m + 1]}({p_1},\xi ,t) - F_2^{[m + 1]}( {p,\xi + \eta ,t} ) \\
			= {F'_2}(V(p,t))[J_\xi ^{m + 1}W({p_1},t) - J_{\xi + \eta }^{m + 1}W(p,t)] + O_{{\theta _n} + m{\theta _{n - 1}} + \varepsilon }^{H_0^1}({({\| \eta \|_{H_0^1}} + {\| \xi \|_{H_0^1}})^{m + 1 + \varepsilon }}),
		\end{multline*}
		where $\xi \in {P_{{N_n}}}(H_0^1)$, ${p_1},p \in {P_{{N_n}}}{\mathcal{M}_{{N_1}}}$, $\varepsilon > 0$ and $\eta : = {p_1} - p$.
		\par
		\textbf{Step 5.} With the necessary conditions established, we can move on to completing the proof for the $( {m + 1} )$th order compatibility condition.
		\par
		Let $R( t ): = J_\xi ^{m + 1}W( {{p_1},t} ) - J_{\xi + \eta }^{m + 1}W( {p,t} )$. From Steps 3, 4, and equation \eqref{extension Taylor jets}, it follows that $R( t )$ solves the problem:
		\begin{equation*}
			\left\{
			\begin{array}{cr}
				\multicolumn{2}{c}{{\partial _t}R + AR - ({F'_1}(V(p,t))R + {F'_2}(V(p,t))R) = O_{{\theta _n} + m{\theta _{n - 1}} + \varepsilon }^H({({\left\| \eta \right\|_{H_0^1}} + {\left\| \xi \right\|_{H_0^1}})^{m + 1 + \varepsilon }})} \\
				& \quad + O_{{\theta _n} + m{\theta _{n - 1}} + \varepsilon }^{H_0^1}({({\left\| \eta \right\|_{H_0^1}} + {\left\| \xi \right\|_{H_0^1}})^{m + 1 + \varepsilon }}), \\
				\multicolumn{2}{c}{\left. P_{{N_n}}R \right|_{t = 0} = 0.}
			\end{array}
			\right.
		\end{equation*}
		\par
		By applying Lemma \ref{my main lemma}, we obtain the result:
		\begin{equation}
			J_\xi ^{m + 1}W({p_1},t) - J_{\xi + \eta }^{m + 1}W(p,t) = O_{{\theta _n} + m{\theta _{n - 1}} + \varepsilon }^{H_0^1}({({\| \eta \|_{H_0^1}} + {\| \xi \|_{H_0^1}})^{m + 1 + \varepsilon }}).
		\end{equation}
		In particular, substituting $m = n - 1$ into the above equation yields
		\begin{equation*}
			\begin{array}{*{20}{c}}
				{{{\| {J_\xi ^nW({p_1},t) - J_{\xi + \eta }^nW(p,t)} \|}_{L_{{\theta _n} + (n - 1){\theta _{n - 1}} + \varepsilon }^2({\mathbb{R}_ - },H_0^1)}} \le C{{({{\| \eta \|}_{H_0^1}} + {{\| \xi \|}_{H_0^1}})}^{n + \varepsilon }},}\\
				{{{\| {J_\xi ^nW({p_1},t) - J_{\xi + \eta }^nW(p,t)} \|}_{{C_{{\theta _n} + (n - 1){\theta _{n - 1}} + \varepsilon }}({\mathbb{R}_ - },H_0^1)}} \le C{{({{\| \eta \|}_{H_0^1}} + {{\| \xi \|}_{H_0^1}})}^{n + \varepsilon }}}
			\end{array}
		\end{equation*}
		for all ${p_1},p \in {P_{{N_n}}}{\mathcal{M}_{{N_1}}}$, $\eta : = {p_1} - p$ and $\xi \in H_0^1$.
		\par
		Finally, by combining the ``Taylor jet" $J_\xi ^{n}W( {p,t} )$ with equation \eqref{closeness}, we complete the proof of Theorem \ref{my theorem}.
	\end{proof}
	
	\section{$C^{n,\varepsilon }$-smooth IF for the Burgers equation}\label{Section 5}
	\noindent
	\par
	The primary goal of this section is to construct the ${C^{n,\varepsilon }}$-smooth IF for equation \eqref{main equation}. Based on the equivalence between the original equation \eqref{main equation} and the transformed equation \eqref{Transformed Equation 2}, we can instead construct the ${C^{n,\varepsilon }}$-smooth extension of IM for equation \eqref{Transformed Equation 2}, and then use spectral projection to obtain the desired ${C^{n,\varepsilon }}$-smooth IF. We now demonstrate that the transformed Burgers equation \eqref{Transformed Equation 2} satisfies the assumptions of Theorem \ref{my theorem}.
	\par
	Let $F_1: = {\mathcal I}_1:H_0^1 \to H$ and ${F_2}: = {{\mathcal I}_2}:H_0^1 \to H_0^1$. By Lemma \ref{Smoothness of the nonlinear terms}, we have ${F_1} \in C_b^{n + 1}(H_0^1,H)$ and ${F_2} \in C_b^{n + 1}(H_0^1,H_0^1)$. Furthermore, for any $K \ge {K_0}(\mathcal{R})$, the maps ${F_1}$ and ${F_2}$ satisfy the following estimates:
	\begin{equation*}
		\begin{array}{*{20}{c}}
			{{{\| {{F'_1}(v)} \|}_{{\mathcal L}(H_0^1,H)}} \le C{K^{ - 1/2}} = :{L_1},}&{{{\| {{F'_2}(v)} \|}_{{\mathcal L}(H_0^1,H_0^1)}} \le {C_K} = :{L_2},}
		\end{array}
	\end{equation*}
	where the constant $C$ depends on $\mathbf{m}$ but is independent of $K$, while the constant $C_K$ depends on $\mathbf{m}$ and $K$, and exhibits monotonic growth as $K$ increases. The linear operator in equation \eqref{transformed Burgers equation} is defined as $A := {\partial _{xx}}$ endowed with Dirichlet boundary conditions. This operator is positive self-adjoint, and its inverse is compact. Its eigenvalues are given by ${\lambda _N} = {N^2}$, where $N \in {\mathbb{N}^ + }$. We are now ready to prove Theorem \ref{main theorem}.
	\begin{proof}[Proof of Theorem \ref{main theorem}] To construct the smooth extension of the IM for the transformed Burgers equation \eqref{transformed Burgers equation}, it suffices to verify that the spectral gap condition \eqref{my spectral gap condition} is satisfied. Indeed, we will demonstrate that the following stronger spectral gap condition holds:
		\begin{equation}
			\begin{array}{*{20}{c}}
				{\frac{{{\lambda _{{N_i} + 1}} - {\lambda _{{N_i}}} - (n - 1){\lambda _{{N_{i - 1}}}}}}{{2\lambda _{{N_i} + 1}^{1/2}(n + 1)}} > {L_1},}&{\frac{{{\lambda _{{N_i} + 1}} - {\lambda _{{N_i}}} - (n - 1){\lambda _{{N_{i - 1}}}}}}{{2(n + 1)}} > {L_2},}&{i = 1,2, \cdot \cdot \cdot n,}&{{\lambda _{{N_0}}} = 0.}
			\end{array}
		\end{equation}
		Clearly, if there exists a sequence $\{ {\lambda _{{N_i}}}\} _{i = 1}^n$ satisfying the above conditions, then the original spectral gap condition \eqref{my spectral gap condition} automatically holds. Next, we proceed to verify that such a sequence can be constructed.
		\par
		Since $L_1=CK^{-1/2}$ and
		\begin{equation*}
			\begin{array}{*{20}{c}}
				{\frac{{{\lambda _{N + 1}} - {\lambda _N}}}{{2\lambda _{N + 1}^{1/2}(n + 1)}} > \frac{{{\lambda _2} - {\lambda _1}}}{{4\lambda _2^{1/2}(n + 1)}},}&{\forall N \in {\mathbb{N}^ + },}
			\end{array}
		\end{equation*}
		we choose $K \gg {K_0}(\mathcal{R})$ such that 
		\begin{equation}
			\frac{{{\lambda _2} - {\lambda _1}}}{{4\lambda _2^{1/2}(n + 1)}} > C{K^{ - 1/2}} = {L_1}.
		\end{equation}
		\par
		On the other hand, since $K$ is fixed, $L_2 = C_K$ is therefore also fixed and as the spectral gap satisfies $\lambda_{N+1} - \lambda_N \sim N$, there exists an integer $\mathcal{N}$ such that:
		\begin{equation*}
			\frac{{{\lambda _{{\mathcal N} + 1}} - {\lambda _{\mathcal N}}}}{{2(n + 1)}} > {C_K} = {L_2}.
		\end{equation*}
		Thus, for all $N > \mathcal{N}$, the following inequalities hold:
		\begin{equation}
			\begin{array}{*{20}{c}}
				{\frac{{{\lambda _{N + 1}} - {\lambda _N}}}{{2\lambda _{N + 1}^{1/2}(n + 1)}} > {L_1},}&{\frac{{{\lambda _{N + 1}} - {\lambda _N}}}{{2(n + 1)}} > {L_2}.}
			\end{array}
		\end{equation}
		Moreover, utilizing the divergence property of the spectral gap, i.e.,
		\begin{equation*}
			\mathop {\lim }\limits_{N \to \infty }(\lambda _{N + 1} - \lambda _N) = \infty,
		\end{equation*}
		we can sequentially select indices:
		\begin{equation}
			\mathcal{N} < N_1 < N_2 < \cdot \cdot \cdot < N_n,
		\end{equation}
		such that the stronger spectral gap condition is satisfied. Specifically, if for some candidate value $N_i$ the conditions:
		\begin{equation*}
			\begin{array}{*{20}{c}}
				{\frac{{{\lambda _{{N_i} + 1}} - {\lambda _{{N_i}}} - (n - 1){\lambda _{{N_{i - 1}}}}}}{{2\lambda _{{N_i} + 1}^{1/2}(n + 1)}} > {L_1},}&{\frac{{{\lambda _{{N_i} + 1}} - {\lambda _{{N_i}}} - (n - 1){\lambda _{{N_{i - 1}}}}}}{{2(n + 1)}} > {L_2},}&{{\lambda _{{N_0}}} = 0,}
			\end{array}
		\end{equation*}
		are not satisfied, we proceed to the next available eigenvalue. Due to the unbounded growth of the spectral gap, a sequence $\{ {\lambda _{{N_i}}}\} _{i = 1}^n$ satisfying the conditions must eventually exist. This ensures that the spectral gap condition \eqref{my spectral gap condition} holds. The conclusion of Theorem \ref{main theorem} then follows from Theorem \ref{my theorem}.
	\end{proof}
	\par
	By applying the smooth diffeomorphism, we can see that Theorem \ref{main theorem} indicates that the long time behavior of Burgers equation \eqref{main equation} can be completely determined by the following smooth first-order ODE:
	\begin{equation*}\label{new IF}
		\begin{array}{*{20}{c}}
			{\frac{d}{{dt}}{v_{{N_n}}} - {\partial _{xx}}{v_{{N_n}}} = {P_{{N_n}}}({\mathcal{I}_1}({v_{{N_n}}} + {\widetilde M_{{N_n}}}({v_{{N_n}}}))) + {P_{{N_n}}}({\mathcal{I}_2}({v_{{N_n}}} + {\widetilde M_{{N_n}}}({v_{{N_n}}}))),}&{v = {\mathcal V}(u),}
		\end{array}
	\end{equation*}
	where ${\mathcal{I}_1}$ and ${\mathcal{I}_2}$ are defined in \eqref{Transformed Equation 2}.
	\par
	\vspace{0.5cm}
	\noindent{\bf Statements and Declarations}
	On behalf of all authors, the corresponding author states that there is no conflict of interest. There is no associated data used in this paper.
	
	\begin{comment}
		This paper overcomes the obstacle posed by the convective term and establishes a ${ {{C^{n,{\varepsilon} }}} |_{\{ n \ge 2,{\varepsilon} \in (0,1)\} }}$-smooth extension of the inertial manifold for the one-dimensional Burgers equation. This construction demonstrates that the dynamical behavior of the one-dimensional Burgers equation can be completely determined by a finite-dimensional system of explicit smooth ODEs. To overcome the regularity loss caused by the convective term, we transform the Burgers equation into a new form via a diffeomorphic transformation. Based on this transformed equation, we introduce an abstract parabolic framework with two distinct nonlinearities. With the aid of the Whitney Extension Theorem (A. Kostianko and S. Zelik in [Anal. PDE, 2024]), we establish a new spectral gap condition sufficient to guarantee the smooth extension of the inertial manifold for such abstract parabolic equations. This condition is then verified for the transformed Burgers equation, leading to the successful construction of the smooth inertial manifold extension for the Burgers equation. Our result not only elucidates the long-time behavior of the Burgers equation, but also provides the first result of constructing the smooth extension of inertial manifolds under the condition of regularity loss due to the convective term. We therefore believe this work will be of interest to researchers in PDEs and dynamical systems.
	\end{comment}
\end{document}